\documentclass[12pt]{amsart}

\usepackage{amsmath}
\usepackage{amsthm}
\usepackage{amsfonts}
\usepackage{amssymb}
\usepackage{bbm}
\usepackage{tikz}
\usepackage{pgf}
\usepackage{graphicx}
\usepackage{mathtools}
\usepackage{hyperref}
\usepackage{ytableau}
\usepackage[margin = 1in]{geometry} 

\DeclareMathOperator{\Forall}{\;\forall}
\DeclareMathOperator{\st}{\,|\,}

\DeclareMathOperator{\C}{\mathbb{C}}
\DeclareMathOperator{\R}{\mathbb{R}}
\DeclareMathOperator{\Z}{\mathbb{Z}}

\newtheorem{theorem}{Theorem}
\numberwithin{theorem}{section}
\newtheorem{lemma}[theorem]{Lemma}
\newtheorem{proposition}[theorem]{Proposition}
\newtheorem{corollary}[theorem]{Corollary}

\theoremstyle{definition}

\newtheorem{example}[theorem]{Example}
\newtheorem{remark}[theorem]{Remark}
\newtheorem{question}[theorem]{Question}

\usepackage[backend=bibtex]{biblatex}
\title{Ferrers Graphs, D-Permutations, and Surjective Staircases}
\author[Lazar]{Alexander Lazar}
\address{Department of Mathematics, KTH Royal Institute of Technology, Stockholm 114 28, Sweden}
\email{alelaz@kth.se}
\date{October 27, 2021}

\begin{document}
\begin{abstract}
We introduce a new family of hyperplane arrangements inspired by the homogenized Linial arrangement (which was recently introduced by Hetyei), and show that the intersection lattices of these arrangements are isomorphic to the bond lattices of Ferrers graphs. Using recent work of Lazar and Wachs we are able to give combinatorial interpretations of the characteristic polynomials of these arrangements in terms of permutation enumeration. For certain infinite families of these hyperplane arrangements, we are able to give generating function formulas for their characteristic polynomials. To do so, we develop a generalization of Dumont's surjective staircases, and introduce a polynomial which enumerates these generalized surjective staircases according to several statistics. We prove a recurrence for these polynomials and show that in certain special cases this recurrence can be solved explicitly to yield a generating function. We also prove refined versions of several of these results using the theory of complex hyperplane arrangements. 
\end{abstract}
\maketitle
\tableofcontents

\section{Introduction}
The aim of this paper is to generalize recent results of Lazar and Wachs \cite{Type_A_Paper, LaWa} on the combinatorics of the \emph{homogenized Linial arrangement} in several different directions.

The homogenized Linial arrangement, which was introduced by Hetyei in \cite{Alternation_Acyclic}, is the real hyperplane arrangement $\mathcal{H}_{2n-1}$ in
$$\{(x_1,\dots,x_{n+1},y_1,\dots,y_{n})\st x_i,\,y_i \in \R\} = \R^{2n+1}$$
given by
$$\mathcal{H}_{2n-1} = \{x_i -x_j = y_i \st 1\leq i < j \leq n+1\}.$$
Hetyei showed that the number of regions of $\mathcal{H}_{2n-1}$ is equal to the \emph{median Genocchi number} $h_{n}$. The median Genocchi numbers and their partner sequence the \emph{Genocchi numbers} $g_n$ are classical sequences that appear in several areas of combinatorics (see \cite{Interpretations_Combinatoires, Alternation_Acyclic, Type_A_Paper, Randrianarivony_Du_Fo_Poly, sund, Zeng_Du_Fo_poly}, among others) and number theory (e.g., \cite{Barsky_Dumont}).

Intersecting $\mathcal{H}_{2n-1}$ with the subspace $y_1 = y_2 = \cdots = y_n = 0$ yields the braid arrangement
$$\{x_i - x_j = 0 \st 1 \leq i < j \leq {n+1}\}$$
in $\R^{n+1}$, while intersecting it with $y_{1} = y_2 = \cdots = y_n = 1$ yields the \emph{Linial arrangement}
$$\{x_i - x_j = 1 \st 1\leq i < j \leq n+1\}$$
in $\R^{n+1}$. We can therefore think of $\mathcal{H}_{2n-1}$ as being defined by a particular choice of homogenization of the defining equations for the hyperplanes in the braid or Linial arrangements. 

In \cite{Type_A_Paper}, Lazar and Wachs studied this arrangement further. They showed that the intersection lattice of $\mathcal{H}_{2n-1}$ is isomorphic to the bond lattice of a bipartite graph $\Gamma_{2n}$ and used that interpretation to give a combinatorial interpretation of the coefficients of the characteristic polynomial $\chi_{\mathcal{L}(\mathcal{H}_{2n-1})}(t)$ of $\mathcal{H}_{2n-1}$ in terms of a new class of permutations they called \emph{D-permutations}. 

Lazar and Wachs showed that the D-permutations on $[2n]$ are bijection with certain elements of the class of \emph{surjective staircases} (due to Dumont \cite{Interpretations_Combinatoires}). Using a generating function result of Randrianarivony \cite{Randrianarivony_Du_Fo_Poly} and Zeng \cite{Zeng_Du_Fo_poly} for surjective staircases, Lazar and Wachs proved (\cite[Theorem 5.5]{Type_A_Paper}) a generating function formula for $\chi_{\mathcal{L}(\mathcal{H}_{2n-1})}(t)$:

\begin{equation}\label{introgenchareq} \sum_{n\geq 1} \chi_{\mathcal L(\mathcal H_{2n-1})}(t) \, x^n =  \sum_{n\geq 1}\frac{  (t-1)_{n-1} (t-1)_{n} \,x^n}{\prod_{k=1}^n(1-k(t-k)x)},
\end{equation}
where $(a)_{n}$ denotes the falling factorial  $a(a-1)\cdots (a-(n-1))$.

The various perspectives used in \cite{Type_A_Paper} --- hyperplane arrangements, $\Gamma_{2n}$ and D-permutations, and surjective staircases --- suggest several possible avenues to generalize the results of that paper. In this paper, we consider three of them.

In the course of \cite{Type_A_Paper}, the authors found it useful to consider a larger class of graphs $\Gamma_V$ where $V \subseteq [2n]$ for any $n$. One possible direction for generalizing of the results of \cite{Type_A_Paper} would therefore be to study the combinatorics of $\Gamma_V$ for arbitrary $V$. The study of these $\Gamma_V$ was initiated in \cite{Type_A_Paper}, where the authors showed that the coefficients of the characteristic polynomial of the bond lattice of $\Gamma_V$ can be interpreted in terms of the D-permutations on $V$. The authors also noted that $\Gamma_{2n}$ belongs to a class of graphs known as \emph{Ferrers graphs}, and remarked that the same was true for all $\Gamma_V$.

Another possible direction for generalizing \cite{Type_A_Paper} is suggested by the theory of hyperplane arrangements. Beyond the braid and Linial arrangements, there are a number of other \emph{deformations} of the braid arrangement (studied in detail in \cite{Defcox}), such as the \emph{Shi arrangement}
$$\{x_i - x_j = 0,1 \st 1\leq i < j \leq n\},$$
\emph{semiorder arrangement}
$$\{x_i - x_j = \pm 1 \st 1\leq i <j \leq n\},$$
and \emph{Catalan arrangement}
$$\{x_i - x_j = -1,0,1 \st 1\leq i < j \leq n\}.$$
Therefore, one might wish to study the combinatorics of homogenizations of these deformations. Let $\nu = (\nu_1,\dots,\nu_n)$ be a weak composition of $m$. That is, $\nu$ is an ordered $n$-tuple of nonnegative integers whose sum is $m$. For all such $\nu$, we define the \emph{homogenized $\nu$-arrangement} $\mathcal{H}_{\nu}$ to be
$$\{x_i - x_j = y_i^{(\ell)} \st 1\leq i < j \leq n+1, 1\leq \ell \leq \nu_i\}.$$
For instance, taking $\nu = (2,2,\dots,2)$ gives us the arrangement
$$\{x_i - x_j = y_i^{(1)},y_i^{(2)} \st 1\leq i < j \leq n+1\}.$$
Intersecting this arrangement with the subspace where $y_1^{(1)} = y_2^{(1)} = \cdots = y_n^{(1)} = 0$ and $y_1^{(2)} = y_2^{(2)} = \cdots = y_n^{(2)} = 1$ yields the Shi arrangement.

The third possible direction is to generalize the notion of surjective staircases. Surjective staircases are surjective functions $F: [2n] \to \{2,4,\dots,2n\}$ that satisfy $F(x) \geq x$. A key tool used in \cite{Type_A_Paper} to prove Equation (\ref{introgenchareq}) is the \emph{generalized Dumont--Foata polynomial} $\Lambda_{2n}$, which enumerates the surjective staircases with domain $[2n]$ with respect to six statistics. Randrianarivony \cite{Randrianarivony_Du_Fo_Poly} and Zeng \cite{Zeng_Du_Fo_poly} proved a recurrence and a generating function formula for the $\Lambda_{2n}$, and it is this formula which eventually yields Equation (\ref{introgenchareq}). One natural generalization would be to consider surjective functions $F: V \to \{\text{even elements of } V\}$ which satisfy $F(x) \geq x$ for any finite $V \subset \Z_{>0}$. Given the role that the generalized Dumont--Foata polynomials play in \cite{Type_A_Paper}, we also wish to define analogous polynomials for domains other than $[2n]$.

The pleasant surprise of these three generalizations is that they remain very closely related. In this paper, we show that for any $\nu$ the intersection lattice of $\mathcal{H}_{\nu}$ is isomorphic to the bond lattice of some $\Gamma_V$. Hence, the coefficients of the characteristic polynomial of $\mathcal{H}_{\nu}$ have an interpretation in terms of the D-permutations on $V$. We then develop the theory of surjective staircases with domain $V$: we define a generalization $\Lambda_V$ of the generalized Dumont--Foata polynomials $\Lambda_{2n}$ and prove a recurrence for the $\Lambda_V$ generalizing the Randrianarivony--Zeng recurrence for $\Lambda_{2n}$.

For certain infinite families of homogenized $\nu$-arrangements, we are able to solve our new recurrence explicitly to obtain generating function formulas for their corresponding $\Lambda_V$.

For one such family that we denote $\mathcal{H}_{n,k}$, this new machinery lets us prove an analog of Equation (\ref{introgenchareq}):
\begin{equation}\label{introkstepgeneq}
\sum_{n\geq 1}\chi_{\mathcal{L}(\mathcal{H}_{n,k})}(t)u^n = \sum_{n\geq 1} \frac{(t-1)_{n-1}\left((t-1)_{n}\right)^ku^{n}}{\prod_{i=1}^{n}\left(1-i(t-i)^{k}u\right)},
\end{equation}
where $(a)_n = a(a-1)\cdots(a-(n-1))$.

These techniques also let us derive a generating function for the chromatic polynomial of the complete bipartite graph $K_{n,k}$:

\begin{equation}
\sum_{n\geq 1}\text{ch}(K_{n,k})(t)u^n =\sum_{n\geq 1}\frac{(t)_{n}\left[(t-n)^k + (n-1)(t-(n-1))^{k-1}\right]u^n}{\prod_{i=0}^{n-1}(1-iu)}.
\end{equation}

The rest of this paper is structured as follows.

\begin{itemize} 
\item In Section 2 we state some preliminary definitions and results. 

\item In Section 3 we give a detailed proof of the equivalence between Ferrers graphs and the graphs $\Gamma_V$ (remarked without proof in \cite[Remark 3.1]{Type_A_Paper}). This equivalence is interesting in its own right, but the language of Ferrers graphs is also convenient for many of our later results.

\item In Section $4$ we formally introduce the hyperplane arrangements $\mathcal{H}_{\nu}$. We that show each such arrangement is isomorphic to the graphic arrangement of a Ferrers graph, and moreover that the arrangement of any Ferrers graph can be obtained in this way. 

\item In Section 5 we give generating function formulas for the characteristic polynomials of two infinite families of Ferrers graphs, and also give refined versions of these generating functions arising from complex hyperplane arrangements (\`{a} la \cite{LaWa}). 

\item In Section 6 we develop the theory of surjective staircases with arbitrary domains, as well as the generalization $\Lambda_V$ of the generalized Dumont--Foata polynomial $\Lambda_{2n}$. We state and prove a recurrence for $\Lambda_V$ that extends the Randrianarivony--Zeng recurrence for $\Lambda_{2n}$.

\item In Section 7 we prove the generating function formulas from Section 5. 

\item In Section 8 we outline an extension of the results of Sections 6 and 7 to complex hyperplane arrangements, and prove the refined generating function formulas of Section 5. 

\item In Section 9 we give some final remarks and questions for further study.
\end{itemize}

Many of the results of this paper were first announced in the extended abstract \cite{Extended_Abstract} and appeared in detail in the author's Ph.D. dissertation \cite[Chapter 5]{Alex_Thesis}.

\section{Preliminaries}
\subsection{Hyperplane Arrangements and Geometric Lattices}
A \emph{real hyperplane arrangement} is a finite collection $\mathcal{H}$ of hyperplanes in $\R^d$. The complement $\R^d\setminus \mathcal{H}$ consists of finitely many connected components, which are called the \emph{regions of $\mathcal{H}$}. The combinatorial data of $\mathcal{H}$ can be collected into the \emph{intersection poset} $\mathcal{L}(\mathcal{H})$, which consists of the intersections of hyperplanes in $\mathcal{H}$ (viewed as affine subspaces of $\R^d$) ordered according to reverse containment. We define the \emph{rank} $\mathrm{rk}(\mathcal{H})$ of $\mathcal{H}$ to be the length of $\mathcal{L}(\mathcal{H})$. 

If all of the hyperplanes in $\mathcal{H}$ have a nonempty common intersection, we say $\mathcal{H}$ is \emph{central}. The intersection poset of a central arrangement is a geometric lattice, and if $\mathcal{H}$ is non-central its intersection poset is a geometric semilattice (in the sense of \cite{geometric_semilattices}).

One important invariant of a hyperplane arrangement is the \emph{characteristic polynomial} $\chi_{\mathcal{L}(\mathcal{H})}(t)$ of its intersection poset. If $P$ is any ranked poset with bottom element $\hat{0}$, its characteristic polynomial is defined by
$$\chi_{P}(t) = \sum_{X \in P}\mu_{P}(\hat{0},X)t^{\mathrm{rk}(P)-\mathrm{rk}(X)},$$

where $\mu_{P}$ is the M\"{o}bius function of $P$.

A well-known formula of Zaslavsky \cite{Facing_Up_Arrangements} relates the number of regions $r(\mathcal{H})$ of $\mathcal{H}$ to the characteristic polynomial of $\mathcal{L}(\mathcal{H})$:
\begin{equation}
r(\mathcal{H}) = (-1)^{\mathrm{rk}(\mathcal{H})}\chi_{\mathcal{L}(\mathcal{H})}(-1).
\end{equation} 

\subsection{Graphs and Bond Lattices}

Associated to any finite graph $G$ with vertex set $V$ and edge set $E$ is the \emph{graphic hyperplane arrangement} $\mathcal{H}_{G}$, which is given by

$$\mathcal{H}_{G} = \{x_i - x_j = 0 \st \{i,j\} \in E\}.$$

The intersection lattice of $\mathcal{H}_G$ is called the \emph{bond lattice of $G$} and is denoted $\Pi_G$. An equivalent description of $\Pi_G$ can be given as a subposet of the lattice $\Pi_V$ of set partitions of $V$. A set partition $\pi = B_1|B_2|\cdots|B_k$ of $V$ belongs to $\Pi_G$ if and only if the induced subgraph $G|_{B_i}$ is connected for all $i$.

The bond lattice of $G$ also determines the chromatic polynomial $\mathrm{ch}(G)(t)$ in the sense that
$$\mathrm{ch}(G)(t) = t\cdot\chi_{\Pi_G}(t).$$

\subsection{The Homogenized Linial Arrangement and D-Permutations}
In \cite{Type_A_Paper}, Lazar and Wachs considered a family of bipartite graphs $\Gamma_V$ for $V$ a finite subset of $\Z_{>0}$. Given such a $V$, $\Gamma_V$ is the graph on vertex set $V$ with an edge $\{2i-1,2j\}$ whenever $2i-1 < 2j$. If $V = [2n]$, we omit the brackets and write $\Gamma_{2n}$. 

A permutation $\sigma \in \mathfrak{S}_{V}$ is a \emph{D-permutation} if $\sigma(i) \geq i$ whenever $i$ is odd and $\sigma(i) \leq i$ whenever $i$ is even. We write $\mathcal{D}_V$ for the set of $D$-permutations of $V$.

\begin{example}
The elements of $\mathcal{D}_4$ are given below.
\begin{center}
\begin{tabular}{cccc}
 $(1)(2)(3)(4)$ & $(1,2)(3)(4)$ & $(1,2)(3,4)$ & $(1,4)(2)(3)$\\
 $(1,4,2)(3)$ & $(1,3,4)(2)$ & $(3,4)(1)(2)$ & $(1,3,4,2)$
\end{tabular}
\end{center}
\end{example}

In \cite{Type_A_Paper}, the authors compute the coefficients of $\chi_{\Pi_{\Gamma_V}}$ in terms of D-permutations:

\begin{theorem}[{\cite[Theorem 3.5]{Type_A_Paper}}] \label{mobth}
	Let $V$ be a finite subset of $\Z_{>0}$.  For  all $\pi \in  \Pi_{\Gamma_{V}}$,
	$$(-1)^{|\pi|}\mu_{\Pi_{\Gamma_{V}}}(\hat{0},\pi) = |\{\sigma \in \mathcal{D}_{V} \st \mathrm{cyc}(\sigma) = \pi\} |,$$
	where $\mathrm{cyc}(\sigma)$ is the set partition whose blocks are comprised of the elements of the cycles of $\sigma$. 
	Consequently, \begin{equation} \label{chardumeq}  \chi_{\Pi_{\Gamma_V}}(t) = 
		\sum_{k=1}^{2n} s_D(V,k) t^{k-1},\end{equation}
	where $(-1)^{k}s_D(V,k)$ is equal to  the number of D-permutations on $A$ with exactly $k$ cycles.
\end{theorem}

\subsection{Ferrers Graphs}
In \cite{Enumerative_Ferrers_Graphs}, Ehrenborg and van Willigenburg introduced a family of graphs called \emph{Ferrers graphs}. Since their introduction, they have been studied in a variety of contexts such as combinatorial commutative algebra \cite{Monomial_Toric_Ideals}, chip-firing \cite{Selig_Smith_Steingrimsson}, matroid theory \cite{Biconed_Graphs}, and in terms of graph invariants like their Boolean complexes \cite{Boolean_Complex_Ferrers}. A Ferrers graph is a bipartite graph on vertex partition $R = \{r_1,\dots, r_n\}$ and $C = \{c_1,\dots, c_m\}$ such that if $\{r_i,c_j\}$ is an edge, then $\{r_{\ell},c_k\}$ is an edge for all $i\leq \ell \leq n$ and $1 \leq k \leq j$, and $\{r_1,c_1\}$ and $\{r_n,c_m\}$ are edges.\footnote{Ehrenborg and van Willigenburg's definition is slightly different from ours --- they reverse the indexes of the $r$'s.} To any Ferrers graph $G$, we can associate an integer partition $\lambda = (\lambda_n,\dots,\lambda_1)$, where $\lambda_i$ is the degree of $r_i$ in $G$. Note that we are distinguishing between $R$ and $C$ in this definition.

If we draw the Ferrers diagram for $\lambda$, label the rows $r_1$ through $r_n$ from bottom to top, and label the columns $c_1$ through $c_m$ from left to right, then the Ferrers diagram for $\lambda$ will have a square in the row labeled $r_i$ and column labeled $c_j$ if and only if $\{r_i,c_j\}$ is an edge of $G$.

\begin{example}
The graph $\Gamma_6$ on vertex set $R = \{2,4,6\}$ and $C = \{1,3,5\}$ is shown below:

\begin{center}
\includegraphics{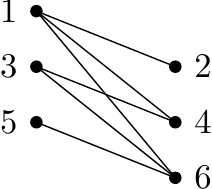}
\end{center}

It is not hard to see that $\Gamma_6$ is a Ferrers graph with associated partition $\lambda = (3,2,1)$, whose Ferrers diagram is seen below:
\begin{center}
\begin{ytableau}
\none[] & \none[1] & \none[3] & \none[5]\\
\none[6] & \; & \; & \;\\
\none[4] & \; & \;\\
\none[2] & \;
\end{ytableau}
\end{center}
\end{example}

The following proposition is an immediate consequence of the Ferrers diagram description of the edges of a Ferrers graph.

\begin{proposition}\label{FerrersNameProp}
If $G$ and $H$ are Ferrers graphs, then $G \cong H$ if and only if $G$ and $H$ have the same associated partition, or have conjugate associated partitions. 
\end{proposition}

\section{Ferrers Graphs and \texorpdfstring{$\Gamma_V$}{ΓV}}\label{FerrersGraphsAndGammaSection}
Let $V$ be a finite subset of $\Z_{>0}$. Let $O(V)$ be the set of odd elements of $V$ and let $E(V)$ be the set of even elements of $V$. Suppose that $O(V) = \{2i_1 - 1 < \cdots < 2i_n - 1\}$ and $E(V) = \{2j_1 < \cdots < 2j_m\}$. We define the \emph{partition type} of $V$, denoted $\lambda(V)$, as follows. For each $k \in [m]$, let $\lambda_{k} \coloneqq \#\{a \st 2i_a - 1 < 2j_k\}$. Then $\lambda_m \geq \cdots \geq \lambda_1$, so $\lambda = (\lambda_m,\dots,\lambda_1)$ is an integer partition.

\begin{example}
Suppose that $V = \{1,3,7,9\} \sqcup \{2,4,6,10\}$. Then $\lambda(V) = (4,2,2,1)$. The Ferrers diagram of $(4,2,2,1)$ is shown below, with the rows labeled by the even elements of $V$ and the columns labeled by the odd elements of $V$:
\begin{center}
\begin{ytableau}
\none[] & \none[1] & \none[3] & \none[7] & \none[9]\\
\none[10] & \; & \; & \; & \;\\
\none[6] & \; & \;\\
\none[4] & \; & \;\\
\none[2] & \;
\end{ytableau}.
\end{center}
\end{example}

The following theorem explains our choice of labeling in the previous example.

\begin{theorem}\label{GammaFerrers}
Let $V$ be a finite subset of $\Z_{>0}$ such that $\max V$ is even and $\min V$ is odd. Then $\Gamma_V$ is a Ferrers graph with $R = E(V)$ and $C = O(V)$. Moreover the partition associated to $\Gamma_V$ is $\lambda(V)$. Conversely, given a Ferrers graph $G$, there is a finite subset $V$ of $\Z_{>0}$ such that $\Gamma_{V} \cong G$. 

\end{theorem}

\begin{proof}
Recall that $\Gamma_V$ has an edge between $2i-1$ and $2j$ if and only if $2i-1 < 2j$. We see that $\{2i_1 - 1, 2j_1\}$ is an edge of $\Gamma_V$ (since $\min V = 2i_1-1$), and $\{2i_m - 1, 2j_n \}$ is an edge (since $2j_n = \max V$). Moreover, if $\{2i_a -1, 2j_b\}$ is an edge of $\Gamma_V$ then $\{2i_k -1, 2j_{\ell}\}$ is an edge of $\Gamma_V$ for all $1\leq k \leq a$ and all $b \leq \ell \leq n$, since $2i_k - 1 \leq 2i_a - 1 < 2j_{b} \leq 2j_{\ell}$. Hence, $\Gamma_V$ is a Ferrers graph with $R = E(V)$ and $C = O(V)$, and moreover the associated partition of $\Gamma_V$ is $\lambda(V)$ by the definition of $\lambda(V)$.

Conversely, suppose that $G$ is a Ferrers graph on vertex set $R = \{r_1,\dots,r_n\}$ and $C = \{c_1,\dots, c_m\}$. Associated to $G$ is the integer partition $\lambda$ given by $\lambda_i = \deg r_i$. We construct a set $V_{\lambda}$ from the Ferrers diagram of $\lambda$ as follows. Let $V_0 = \{0\}$.

\begin{enumerate}
\item Start at the southwest corner of the Ferrers diagram for $\lambda$, and walk the path along the lower border of the Ferrers diagram.
\item If the $i$th step is to the east, label the edge we have just traversed with the smallest odd integer $2k-1$ larger than $\max V_{i-1}$ and let $V_i \coloneqq V_{i-1} \cup \{2k-1\}$. 
\item If the $i$th step is to the north, label the edge we have just traversed with the smallest even integer $2k$ larger than $\max V_{i-1}$ and let $V_i \coloneqq V_{i-1} \cup \{2k\}$.
\item Stop when the northeast corner of the Ferrers diagram is reached, after $m+n$ steps. We define $V_{\lambda} \coloneqq V_{n+m}\setminus \{0\}$. 
\end{enumerate} 

Note that $\min V_{\lambda}$ is odd and $\max V_{\lambda}$ is even, since our walk must start with a step to the east and end with a step north. Thus, $\Gamma_{V_{\lambda}}$ is a Ferrers graph by our previous argument. By Proposition \ref{FerrersNameProp}, to show that $\Gamma_{V_{\lambda}} \cong G$, it suffices to show that $\lambda(V_{\lambda}) = \lambda$. Indeed, $\lambda_i$ is the number of squares in the $i$th row from the bottom of the Ferrers diagram for $\lambda$. By construction, the number of squares in the $i$th row from the bottom of the Ferrers diagram is equal to the number of odd elements of $V_{\lambda}$ that are smaller than the $i$th smallest even element $2j_i$ of $V$. But by definition, this quantity is equal to $\lambda(V_{\lambda})_i$. Hence, $\lambda(V_{\lambda}) = \lambda$. 
\end{proof}

Our construction of $V_{\lambda}$ is very similar to the \emph{\textbf{ab}-word} construction in \cite{Enumerative_Ferrers_Graphs}, which Ehrenborg and van Willigenburg use to give a formula for the chromatic polynomial of $G_{\lambda}$.

The following corollary is an immediate consequence of Proposition \ref{FerrersNameProp} and Theorem \ref{GammaFerrers}.

\begin{corollary}\label{LambdaGammaThm}
Let $V$ and $W$ be finite sets of positive integers such that $\min V$, $\min W$ are odd and $\max V$, $\max W$ are even. Then $\Gamma_V \cong \Gamma_W$ if and only if $\lambda(V) = \lambda(W)$.
\end{corollary}

\begin{remark}
Let $G$ be a Ferrers graph. By \cite[Theorem 4.10]{Type_A_Paper} we see that
$$\text{ch}_G(t) = \sum_{\sigma \in \mathcal{D}_V}(-t)^{c(\sigma)},$$
where $c(\sigma)$ is the number of cycles of $\sigma$, and $V$ is the set constructed from $G$ in our proof of Theorem \ref{GammaFerrers}.
\end{remark}

\section{Homogenized \texorpdfstring{$\nu$}{ν}-Arrangements}
Let $\nu = (\nu_1,\nu_2,\dots, \nu_n)$ be a weak composition of $m$. We define the \emph{homogenized $\nu$-arrangement} in
$$(x_1,x_2,\dots,x_{n+1},y_1^{(1)},y_1^{(2)},\dots,y_1^{(\nu_1)},y_2^{(1)},\dots,y^{(\nu_2)}_2,\dots,y_n^{(1)},\dots,y_n^{(\nu_n)}) = \R^{m + n +1}$$
to be
$$\mathcal{H}_{\nu} = \{x_i - x_j = y_i^{(\ell)} \st 1\leq i < j \leq n+1, 1\leq \ell \leq \nu_i\}.$$

Notice that when $\nu = (1,1,\dots,1)$, $\mathcal{H}_{\nu} = \mathcal{H}_{2n-1}$. Our goal for this section is to study these homogenized $\nu$-arrangements by relating them to Ferrers graphs.

Let $\lambda(\nu)$ be the partition $(\nu_1+ \cdots + \nu_n,\dots,\nu_1+ \nu_2,\nu_1)$. We construct a Ferrers graph $G_{\nu}$ from $\nu$ by labeling the columns of the Ferrers diagram for $\lambda(\nu)$ from left to right by 
$$1^{(1)}, 1^{(2)},\dots,1^{(\nu_1)},3^{(1)},3^{(2)},\dots,3^{(\nu_2)},\dots,(2n-1)^{(1)},(2n-1)^{(2)},\dots,(2n-1)^{(\nu_n)}$$
and labeling the rows of the Ferrers diagram from bottom to top by
$$2,4,6, \dots, 2n.$$
Then $G_{\nu}$ is the Ferrers graph on vertex set $C \sqcup R$, where
$$C = \{1^{(1)},\dots,1^{(\nu_1)},3^{(1)},\dots,3^{(\nu_2)},\dots,(2n-1)^{(1)},\dots,(2n-1)^{(\nu_n)}\} $$
and
$$R = \{2,\dots,2n\},$$ 
whose edges are $\{(2i-1)^{(\ell)},2j\}$ for all $2i-1 < 2j$ and $1\leq \ell \leq \nu_i$.

\begin{example}
Let $\nu = (3,3,3)$. The Ferrers diagram for $\lambda(\nu) = (9,6,3)$, together with the labels described above, is shown below.
$$\begin{ytableau}
\none & \none[1^{(1)}] & \none[1^{(2)}] & \none[1^{(3)}] & \none[3^{(1)}] & \none[3^{(2)}] & \none[3^{(3)}] & \none[5^{(1)}] & \none[5^{(2)}] & \none[5^{(3)}]\\
\none[6] & \; & \; & \; & \; & \; & \; & \; & \; & \;\\
\none[4] & \; & \; & \; & \; & \; & \;\\
\none[2] & \; & \; & \; 
\end{ytableau}.$$
\end{example}

\begin{theorem}\label{LambdaBondTh}
For all weak compositions $\nu = (\nu_1,\dots,\nu_n)$ of $m$, there is an invertible $\Z$-linear transformation from $\R^{m+n+1}$ to itself that takes $\mathcal{H}_{\nu}$ to the graphic hyperplane arrangement $\mathcal{A}_{G_{\nu}}$. Consequently, $\mathcal{L}(\mathcal{H}_{\nu}) \cong \Pi_{G_{\nu}}$.
\end{theorem}

\begin{proof}
Let $(e_1,\dots,e_{m+n+1})$ be the standard basis for $\R^{m+n+1}$. For notational convenience, we define vectors $u_i$ and $u_i^{(j)}$ by setting
\begin{align*}(u_1^{(1)},u_1^{(2)},\dots,u_1^{(\nu_1)},&u_2,u_3^{(1)},\dots,u_3^{(\nu_2)},u_4,\dots,u_{2n-1}^{(1)},\dots,u_{2n-1}^{(\nu_n)},u_{2n},u_{m+n+1})\\ &= (e_1,\dots,e_{m+n+1}),\end{align*}
and also define $v_i$ and $v_i^{(j)}$ by setting
\begin{align*}(v_1,v_2,\dots,v_{n+1},v_{n+2}^{(1)},\dots,v_{n+2}^{(\nu_1)},&v_{n+3}^{(1)},\dots,v_{n+3}^{(\nu_2)},\dots,v_{2n+1}^{(1)},\dots,v_{2n+1}^{(\nu_n)})\\ &= (e_1,\dots,e_{m+n+1}).\end{align*}

For all $1\leq i\leq j \leq n$ and all $1\leq \ell \leq \nu_i$, we define the hyperplane
$$K_{i,j}^{(\ell)} = \{w \in \R^{m + n + 1} \st (u_{2i-1}^{(\ell)} - u_{2j}) \cdot w = 0\},$$
and the arrangement
$$\mathcal{K} = \{K_{i,j}^{\ell} \st 1\leq i \leq j \leq n, \; 1\leq \ell \leq \nu_i\}.$$

We observe that $\mathcal{K} \cong \mathcal{A}_{G_{\nu}}$. Indeed, $\mathcal{K}$ is the graphic hyperplane arrangement of the graph on $m + n + 1$ vertices obtained from $G_{\nu}$ by adjoining an isolated vertex $m + n + 1$.

Now, for all $1 \leq i \leq j \leq n$ and $1 \leq \ell \leq \nu_i$, we write
$$H_{i,j}^{\ell} = \{w \in \R^{m + n + 1} \st (v_i - v_{j+1} - v_{n+i+1}^{(\ell)})\cdot w = 0\}.$$
Clearly, 
$$\mathcal{H}_{\nu} = \{H_{i,j}^{\ell} \st 1\leq i \leq j \leq n, \; 1\leq \ell \leq \nu_i\}.$$

Now, consider the linear transformation $\phi$ from $\R^{m + n + 1}$ to itself, given by 
\begin{itemize}
\item $\phi(u_{2i-1}^{\ell}) = v_i - v_{n+i+1}^{(\ell)}$, 
\item $\phi(u_{2j}) = v_{j+1}$,
\item $\phi(u_{m + n + 1}) = v_1$. 
\end{itemize}
To show that this $\Z$-linear transformation is invertible, we define the linear transformation $\tilde{\phi}$ from $\R^{m + n +1}$ to itself on the standard basis as follows: 
\begin{itemize} 
\item $\tilde{\phi}(v_1) = u_{m+n+1}$, 
\item $\tilde{\phi}(v_i) =  u_{2i-2}$ for $2 \leq i \leq n+1$,
\item $\tilde{\phi}(v_{n+2}^{(\ell)}) = u_{\lambda_{n}+n+1} - u_1^{(\ell)}$ for $\ell \in [\nu_1]$, 
\item $\tilde{\phi}(v_{n+ i + 1}^{(\ell)}) = u_{2i-2} - u_{2i-1}^{(\ell)}$ for $1 < i \leq n$ and $\ell \in [\nu_i]$. 
\end{itemize}

We see that
\begin{itemize}
\item $\phi(\tilde{\phi}(v_1)) = \phi(u_{m + n + 1}) = v_1$
\item $\phi(\tilde{\phi}(v_i)) = \phi(u_{2i-2}) = v_i$ for $2\leq i \leq n+1$
\item $\phi(\tilde{\phi}(v_{n+2}^{(\ell)})) = \phi(u_{m + n +1} - u_1^{(\ell)}) = v_1 - (v_1 - v_{n+2}^{(\ell)}) = v_{n+2}^{(\ell)}$ for $\ell \in [\nu_1]$
\item $\phi(\tilde{\phi}(v_{n+i + 1}^{(\ell)})) = \phi(u_{2i-2} - u_{2i-1}^{(\ell)}) = v_i - (v_i - v_{n+i+1}^{(\ell)}) = v_{n+i+1}^{(\ell)}$ for $2\leq i \leq n$, $\ell \in [\nu_i]$.
\end{itemize}

Hence, $\phi$ is invertible, so let $A$ be the matrix of $\phi$ with respect to the standard basis and let $\psi$ be the linear operator on $\R^{m+n+1}$ whose matrix in the standard basis is $(A^{-1})^{T}$. 

We claim that $\psi$ takes each hyperplane $K_{i,j}^{\ell}$ to the corresponding hyperplane $H_{i,j}^{\ell}$. Indeed, suppose $w \in K_{i,j}^{\ell}$ so $(u_{2i-1}^{(\ell)} - u_{2j})\cdot w = 0$. Then
$$\phi(u_{2i-1}^{(\ell)} - u_{2j})\cdot\psi(w) = (v_i - v_{n+i+1}^{(\ell)} - v_{j+1})\cdot\psi(w),$$
and
\begin{align*}
\phi(u_{2i-1}^{(\ell)} - u_{2j})\cdot\psi(w) &= A(u_{2i-1}^{(\ell)} - u_{2j})\cdot ((A^{-1})^{T}w)\\
&= (A^{-1}A(u_{2i-1}^{(\ell)} - u_{2j}))\cdot w\\
&= (u_{2i-1}^{(\ell)} - u_{2j})\cdot w\\
&= 0,
\end{align*}
so $(v_i - v_{n+i+1}^{(\ell)} - v_{j+1})\cdot\psi(w) = 0$. Hence, $\psi(w) \in H_{i,j}^{(\ell)}$, which proves the claim.
\end{proof}

\section{New Generating Function Formulas}\label{GenFunFormulaSection}
In this section we present generating function formulas for two infinite families of Ferrers graphs. The proofs of these formulas rely on techniques that we develop in Section \ref{SurjStaircaseTechniques}, so we defer the proofs to Sections \ref{GenFunProofSec} and \ref{DowlingSec}.

\subsection{\texorpdfstring{$k$}{k}-Staircases}\label{kStaircaseSection}
For any $n,k \geq 1$, we define the \emph{$k$-staircase with $n$ steps} to be the partition $\lambda_k^{(n)} = (nk,(n-1)k,\dots,k)$. The case $n=4$ and $k=3$ is shown below.

$$\begin{ytableau}
\; & \; & \; & \; & \; & \; & \; & \; & \; & \; & \; & \;\\
\; & \; & \; & \; & \; & \; & \; & \; & \;\\
\; & \; & \; & \; & \; & \;\\
\; & \; & \;
\end{ytableau}$$

\begin{theorem}\label{kStaircaseCharGenTh}
Fix $k\geq 1$. For any sequence of subsets $V_1^k,V_2^k,\dots$ of $\Z_{>0}$ with $\lambda(V_n^k) = \lambda^{(n)}_k$, we have the following generating function formula for $\chi_{\Pi_{\Gamma_{V_n^k}}}(t)$:
$$\sum_{n\geq 1}\chi_{\Pi_{\Gamma_{V_n^k}}}(t)u^n = \sum_{n\geq 1} \frac{(t-1)_{n-1}\left((t-1)_{n}\right)^ku^{n}}{\prod_{i=1}^{n}\left(1-i(t-i)^{k}u\right)},$$
where $(x)_n = x(x-1)\cdots(x-(n-1))$.
\end{theorem}

This generating function formula (and equivalently the formula in Corollary \ref{kCharPolyCor}) reduces to \cite[Equation (1.9)]{Type_A_Paper} when $k=1$; notice that when $k=1$, $\Gamma_{V_n^k} \cong \Gamma_{2n}$. Applying Theorem \ref{GammaFerrers} and the fact that $t\chi_{\Pi_G}(t)$ is the chromatic polynomial of $G$, we can restate Theorem \ref{kStaircaseCharGenTh} as follows:

\begin{corollary}\label{kStaircaseCharGenCorollary}
For each $n,k \geq 1$, let $G_{n,k}$ be a Ferrers graph whose associated partition is $\lambda_{k}^{(n)}$.
Then we have the following generating function formula for the chromatic polynomial $\text{ch}_{G_{n,k}}(t)$:
$$\sum_{n\geq 1}\text{ch}_{G_{n,k}}(t)u^n = \sum_{n\geq 1} \frac{(t)_{n}\left((t-1)_{n}\right)^ku^{n}}{\prod_{i=1}^{n}\left(1-i(t-i)^{k}u\right)}.$$
\end{corollary}

As a consequence of Theorem \ref{LambdaBondTh}, we can view the generating function formula of Theorem \ref{kStaircaseCharGenTh} as a generating function formula for the characteristic polynomials of the homogenized $(k,k,\dots,k)$-arrangements, as seen in Equation (\ref{introkstepgeneq}):

\begin{corollary}\label{kCharPolyCor}
Let $\mathcal{H}_{n,k} \coloneqq \mathcal{H}_{(k,k,\dots,k)}$. Then
$$\sum_{n\geq 1}\chi_{\mathcal{L}(\mathcal{H}_{n,k})}(t)u^n = \sum_{n\geq 1} \frac{(t-1)_{n-1}\left((t-1)_{n}\right)^ku^{n}}{\prod_{i=1}^{n}\left(1-i(t-i)^{k}u\right)}.$$
\end{corollary}

If we define an appropriate complex hyperplane arrangement $\mathcal{H}_{n,k}^m$, then by combining the techniques from this paper and \cite[Chapter 4]{Alex_Thesis} (see also the forthcoming paper \cite{LaWa}), we can compute an $m$-analog of the generating function from Theorem \ref{kStaircaseCharGenTh}:

\begin{theorem}\label{kStaircaseDowlingThm}
Let $\mathcal{H}_{n,k}^m$ be the hyperplane arrangement in
$$\{(x_1,\dots,x_n,y_1^{(1)},\dots,y_1^{(k)},\dots,y_n^{(1)},\dots,y_n^{(k)})\st x_i,y_j^{(\ell)} \in \C\} = \C^{nk+n}$$
given by
\begin{align*}\mathcal{H}_{n,k}^m = &\{x_i - \omega^{p}x_j = y_i^{(\ell)} \st 1\leq i <j\leq n, 0\leq p \leq m-1, 1\leq \ell \leq k\}\\ &\cup \{x_i = y^{(\ell)}_i \st 1\leq i \leq n, 1\leq \ell \leq k\},
\end{align*}
where $\omega = e^{\frac{2 \pi i}{m}}$.
Then
\begin{equation}
\sum_{n\geq 1}\chi_{\mathcal{L}(\mathcal{H}_{n,k}^m)}(t)u^{n-1} = \sum_{n\geq 1}\frac{(t-1)_{n-1,m}\left((t-1)_{n,m}\right)^ku^{n-1}}{\prod_{i=0}^{n-1}\left(1-(im+1)(t-(im+1))^ku\right)},
\end{equation}
where $(x)_{n,m} = x(x-m)\cdots(x-(n-1)m)$. 
\end{theorem}
Note that this equation specializes to \cite[Equation (1.11)]{LaWa, Alex_Thesis} when $k=1$.

\subsection{Complete Bipartite Graphs}\label{CompBipSection}
Let $T_n^k \coloneqq \{1,3,5,\dots,2k-1\}\sqcup\{2k,\dots,2(k+n-1)\}$. We note that $\lambda(T_n^k) = \mu_k^{(n)}$ where $\mu_k^{(n)}$ is the rectangular shape $(\underbrace{k,k,\dots,k}_n)$. The partition $\mu_3^{(4)}$ is shown below.

$$\begin{ytableau} \; & \; & \;\\ 
\; & \; & \; \\ 
\; & \;& \;\\ 
\; & \; & \; \end{ytableau}$$

It is clear that the Ferrers graph $\Gamma_{T_n^k}$ is the complete bipartite graph $K_{n,k}$. The chromatic polynomial of $K_{n,k}$ has been well-studied (Swenson gives a closed form expression for it in \cite{Swenson_Chromatic_Bipartite}, and Ehrenborg and van Willigenburg give a different proof in \cite{Enumerative_Ferrers_Graphs}). Using our techniques, we give a generating function formula for its chromatic polynomial.

\begin{theorem}\label{CompBipGenThm}
For all $k\geq 1$, we have
$$\sum_{n\geq 1}\text{ch}(K_{n,k})(t)u^n =\sum_{n\geq 1}\frac{(t)_{n}\left[(t-n)^k + (n-1)(t-(n-1))^{k-1}\right]u^n}{\prod_{i=0}^{n-1}(1-iu)}.$$
\end{theorem}

We can also give an $m$-analog of the generating function for the characteristic polynomial of the complete bipartite graph.

\begin{theorem}\label{CompBipDowlingGenThm}
Let $\mathcal{J}_{n,k}^m$ be the hyperplane arrangement in
$$\{(x_1,\dots,x_n,y^{(1)},\dots,y^{(k)}) \st x_i, y^{(j)} \in \C\} = \C^{n+k}$$
given by
$$\mathcal{J}_{n,k}^m = \{x_1 - \omega^{p}x_j = y^{(\ell)} \st 1 < j \leq n, 0\leq p \leq m-1, 1\leq \ell \leq k\} \cup \{x_1 = y^{(\ell)} \st 1\leq \ell \leq k\},$$
where $\omega = e^{\frac{2 \pi i}{m}}$.

Then $\chi_{\mathcal{J}_{1,k}^m}(t) = (t-1)^k$ and
\begin{align*}\sum_{n\geq 2} &\chi_{\mathcal{J}_{n,k}^m}(t)u^{n-1}\\ &= \sum_{n\geq 2}\frac{(t-1)_{n-1,m}\left[(t-1-m(n-1))^k+(m(n-2)+1)(t-1-m(n-2))^{k-1}\right]u^{n-1}}{\prod_{i=0}^{n-2}\left[1-(mi+1)u\right]}.
\end{align*}
\end{theorem}

\section{General Techniques}\label{SurjStaircaseTechniques}
In this section we will state and prove several technical results generalizing the theory of surjective staircases, which were introduced by Dumont in \cite{Interpretations_Combinatoires} (this theory was used in \cite{Type_A_Paper} to study the homogenized Linial arrangement). These results will be used in Sections \ref{GenFunProofSec} and \ref{DowlingSec} to prove the results of Section \ref{GenFunFormulaSection}.

\subsection{Generalized Surjective Staircases}
Let $S$ be a finite subset of the positive integers with largest element $2n$. We define the \emph{staircase diagram of $S$} to be the Ferrers diagram whose rows are labeled from bottom to top with the even elements of $S$ in increasing order and whose columns are labeled from left to right with all of the elements of $S$ in increasing order, with a cell in the row labeled $2i$ and column labeled $j$ if and only if $j\leq 2i$. Note that the staircase diagram of $S$ is not the same as the Ferrers diagram of the partition $\lambda(S)$ defined in Section \ref{FerrersGraphsAndGammaSection} --- we can recover $\lambda(S)$ from the staircase diagram of $S$ by deleting the columns labeled with the even elements of $S$. 

Let $S'$ be the subset of $S$ obtained by removing $2n$ and all odd elements between $2m$ and $2n$ from $S$, where $2m$ is the second-largest even element of $S$ (if $S$ contains only one even element, $S' = \emptyset$). We can view $S'$ as the set whose staircase diagram is obtained by deleting the top row of the staircase diagram of $S$. Similarly, let $S''$ be the set whose staircase diagram is obtained from the staircase diagram of $S'$ by deleting the top row of its staircase diagram.

\begin{example}
Suppose that $S = \{1,2,4,5,7, 8\}$. The staircase diagram for $S$ is seen below.
\begin{center}
\begin{ytableau}
\none & \none[1] & \none[2] & \none[4] & \none[5] & \none[7] & \none[8]\\
\none[8] & \; & \; & \; & \; & \; & \;\\
\none[4] & \; & \; & \;\\
\none[2] & \; & \;
\end{ytableau}
\end{center}
In this case, $S'= \{1,2,4\}$. The staircase diagram for $S'$ is seen below.
\begin{center}
\begin{ytableau}
\none & \none[1] & \none[2] & \none[4]\\
\none[4] & \; & \; & \;\\
\none[2] & \; & \;
\end{ytableau}
\end{center}
Finally, $S'' = \{1,2\}$. The staircase diagram for $S''$ is seen below.
\begin{center}
\begin{ytableau}
\none & \none[1] & \none[2]\\
\none[2] & \; & \;
\end{ytableau}
\end{center}
\end{example}

We say a function $F: S \to S$ is excedent if $F(x) \geq x$ for all $x \in S$. A \emph{generalized sujective staircase} is a surjective excedent map $F$ from $S$ to the set of even elements of $S$.  Let $\mathcal{X}_S$ be the set of all excedent functions $F: S \to S$ and $\mathcal{E}_S$ be the set of all generalized surjective staircases with domain $S$. 

For all $F \in \mathcal{E}_S$, we define
$$w(F) = x^{\text{mo}(F)}y^{\text{fd}(F)}z^{\text{si}(F)}\bar{x}^{\text{me}(F)}\bar{y}^{\text{fi}(F)}\bar{z}^{\text{sd}(F)},$$
and
$$\Lambda_S(x,y,z,\bar{x},\bar{y},\bar{z}) = \sum_{F \in \mathcal{E}_S}w(F),$$
where the statistics are defined as follows:
\begin{itemize}
	\item $2i-1 \in S'$ is a \emph{surfixed point} of $F$ if $F(2i-1)$ is the least even element of $S$ larger than $2i-1$. A surfixed point $2i-1$ is \emph{isolated} if there is no $j$ with $F(j) = F(2i-1)$ and \emph{doubled} otherwise. We write $\text{si}(F)$ and $\text{sd}(F)$ for the numbers of isolated and doubled surfixed points, respectively.
	
	\item $2i \in S'$ is a \emph{fixed point} of $F$ if $F(2i) = 2i$. A fixed point $2i$ is \emph{isolated} if there is no $j$ with $F(j) = 2i$ and is \emph{doubled} otherwise. We write $\text{fd}(F)$ and $\text{fi}(F)$ for the numbers of isolated and doubled fixed points, respectively.
	
	\item $2i-1 \in S'$ is an \emph{odd maximum} of $F$ if $F(2i-1) = 2n$ and $2i \in S'$ is an \emph{even maximum} of $F$ if $F(2i) = 2n$. We write $\text{mo}(F)$ and $\text{me}(F)$ for the numbers of odd and even maxima of $F$, respectively.
\end{itemize}

Note that elements of $S\setminus S'$ do not count as maxima or (sur)fixed points.

Generalized surjective staircases can be visualized using fillings of staircase diagrams of $S$.

\begin{example}
The following surjective staircase has weight $z\bar{x}\bar{y}$:

\begin{center}
\begin{ytableau}
\none & \none[1] & \none[2] & \none[4] & \none[5] & \none[7] & \none[8]\\
\none[8] & \; & X & \; & X & X & X\\
\none[4] & \; & \; & X\\
\none[2] & X & \;
\end{ytableau}.
\end{center}

\end{example}

When $S = [2n]$, these polynomials specialize to the \emph{generalized Dumont--Foata polynomials} $\Lambda_{2n}$. The generalized Dumont--Foata polynomials were introduced by Dumont in \cite{Dumont_Foata}, and Randrianarivony \cite{Randrianarivony_Du_Fo_Poly} and Zeng \cite{Zeng_Du_Fo_poly} independently proved a generating function formula for them.

\begin{remark}\label{GenSurjStaircaseRemark}
It is clear to see that $\Lambda_S$ depends only on $\lambda(S)$. That is, if $S$ and $T$ are two finite sets of integers with $\lambda(S) = \lambda(T)$, then the polynomials $\Lambda_S$ and $\Lambda_T$ are also equal.
\end{remark}

\begin{theorem}\label{GeneralizedRecursion}
Let $S$ be a finite set of positive integers whose largest element is even. Suppose that $S'\setminus S''$ contains $\ell$ odd elements. Then,
\begin{enumerate}
\item If $|S'| = 0 $, $\Lambda_S = 1$.
\item If $|S'| \geq 1$,
\begin{align*}\Lambda_{S}(x,y,z,\bar{x},\bar{y},\bar{z}) = &(y+\bar{x})(x+\bar{z})^{\ell}\Lambda_{S'}(x+1,y,z,\bar{x}+1,\bar{y},\bar{z})\\
 &+ x^{\ell-1}[x(\bar{y}-y) + \ell \bar{x}(z-\bar{z}) - x\bar{x}]\Lambda_{S'}(x,y,z,\bar{x},\bar{y},\bar{z}).\end{align*}
\end{enumerate}
\end{theorem}

When $\ell = 1$, this is the same as the recurrence for $\Lambda_{2n}$ in terms of $\Lambda_{2n-2}$ from \cite[Theorem 3]{Randrianarivony_Du_Fo_Poly} and \cite[Theorem 4]{Zeng_Du_Fo_poly}. Indeed, our proof generalizes the techniques used in Zeng's proof of \cite[Theorem 4]{Zeng_Du_Fo_poly}. Our proof requires the following lemma.

\begin{lemma}\label{TopRowLemma}
Given a surjective staircase $F \in \mathcal{E}_{S'}$ construct a surjective staircase $\hat{F}_X \in \mathcal{E}_S$ by choosing a proper subset $X$ of $F^{-1}(2m)$, where $2m = \max S'$, and defining $\hat{F}_X$ by
$$\hat{F}_X(a) = \begin{cases} 2n, & a \in X \cup (S\setminus S')\\ F(a), & \text{otherwise} \end{cases}.$$
This map $(F,X) \mapsto \hat{F}_X$ is a well-defined bijection from the set of pairs $(F,X)$ where $F \in \mathcal{E}_{S'}$ and $X \subsetneq F^{-1}(2m)$ to $\mathcal{E}_S$.
\end{lemma}

\begin{proof}
Diagramatically, $\hat{F}_X$ is obtained from $F$ by adding a new top row to the staircase diagram for $F$, moving a proper subset of the filled squares in the top row of $F$ to this new top row, and filling all of the squares corresponding to the elements of $S\setminus S'$.
\end{proof}

\begin{example}
Let $S = \{1,2,4,5,7,8\}$. Then $S' = \{1,2,4\}$. Suppose that $F \in \mathcal{E}_{S'}$ is the surjective staircase pictured below, and $X = \{2\}$.
\begin{center}
\begin{ytableau}
\none & \none[1] & \none[2] & \none[4]\\
\none[4] & \; & *(red) X & X\\
\none[2] & X & \;
\end{ytableau}
\end{center}

Then $\hat{F}_X \in \mathcal{E}_{S}$ is the following surjective staircase.
\begin{center}
\begin{ytableau}
\none & \none[1] & \none[2] & \none[4] & \none[5] & \none[7] & \none[8]\\
\none[8] & \; & X & \; & X & X & X\\
\none[4] & \; & \; & X\\
\none[2] & X & \;
\end{ytableau}
\end{center}
\end{example}

\begin{proof}[Proof of Theorem \ref{GeneralizedRecursion}]
If $|S'| = 0$, then $|\Lambda_S| = 1$ and the weight of that surjective staircase must be $1$.

Now, suppose $|S'| \geq 1$. We proceed by using Lemma \ref{TopRowLemma} to compute $\Lambda_S$ in terms of $\Lambda_{S'}$.

Given a pair $(X,F)$, we wish to compute $w(\hat{F}_X)$ in terms of $w(F)$ and $X$. Let $\text{MO}(F)$ and $\text{ME}(F)$ be the sets of odd and even maxima of $F$, respectively (so $|\text{MO}(F)| = \text{mo}(F)$ and $|\text{ME}(F)| = \text{me}(F)$). Furthermore, let $O(X)$ and $E(X)$ be the numbers of odd and even elements of $X$, respectively.

\textbf{Case 1.} Suppose that $2m \notin X$.

If $X = F^{-1}(2m)\setminus \{2m\}$, then
\begin{align*}w(\hat{F}_X) &= x^{\text{mo}(F)+\ell}y^{\text{fd}(F)}z^{\text{si}(F)}\bar{x}^{\text{me}(F)}\bar{y}^{\text{fi}(F)+1}\bar{z}^{\text{sd}(F)}\\
&= x^{\ell}\bar{y}w(F).\end{align*}

Thus, the set of all such $\hat{F}_X$ contribute $x^{\ell}\bar{y}\Lambda_{S'}$ to $\Lambda_S$.

Now, suppose that $X$ contains exactly $k$ of the $\ell$ odd elements of $S'\setminus S''$ for $0\leq k \leq \ell$, but that $X \subsetneq F^{-1}(2m)\setminus \{2m\}$. Then $O(X) \geq k$, since the new maxima of $\hat{F}_X$ consist of some subset of the odd maxima of $F$ along with the $k$ odd elements of $|T\setminus U|$. Hence we have
\begin{align*}
w(\hat{F}_X) &= x^{O(X)}y^{\text{fd}(F)+1}z^{\text{si}(F)}\bar{x}^{E(X)}\bar{y}^{\text{fi}(F)}\bar{z}^{\text{sd}(F)+(\ell-k)}\\
&= x^{k}y\bar{z}^{\ell-k}\left(x^{O(X)-k}y^{\text{fd}(F)}z^{\text{si}(F)}\bar{x}^{E(X)}\bar{y}^{\text{fi}(F)}\bar{z}^{\text{sd}(F)}\right).
\end{align*}

Any such $X$ can be obtained by choosing subsets $A$ and $B$ of the odd and even maxima of $F$, respectively, and then choosing $k$ elements of $S'\setminus S''$ in $\binom{\ell}{k}$ ways. Thus, when $k<\ell$ we see
\begin{align*}
\sum_{F \in \mathcal{E}_{S'}}&\sum_{\substack{X \subsetneq F^{-1}(2m)\setminus\{2m\}\\ |X\cap (S'\setminus S'')| = k}}w(\hat{F}_X)\\ 
&= \sum_{F \in \mathcal{E}_{S'}}\sum_{\substack{X \subsetneq F^{-1}(2m)\setminus\{2m\}\\ |X\cap (S\setminus S'')| = k}}x^{k}y\bar{z}^{\ell-k}\left(x^{O(X)-k}y^{\text{fd}(F)}z^{\text{si}(F)}\bar{x}^{E(X)}\bar{y}^{\text{fi}(F)}\bar{z}^{\text{sd}(F)}\right)\\
&= y\sum_{F \in \mathcal{E}_{S'}}y^{\text{fd}(F)}z^{\text{si}(F)}\bar{y}^{\text{fi}(F)}\bar{z}^{\text{sd}(F)}\binom{\ell}{k}x^k\bar{z}^{\ell-k}\sum_{A\times B \subseteq \text{MO}(F)\times\text{ME}(F)}x^{|A|}\bar{x}^{|B|}\\
&=y\binom{k}{\ell}x^k\bar{z}^{\ell-k}\sum_{F \in \mathcal{E}_{S'}}y^{\text{fd}(F)}z^{\text{si}(F)}\bar{y}^{\text{fi}(F)}\bar{z}^{\text{sd}(F)}(x+1)^{\text{mo}(F)}(\bar{x}+1)^{\text{me}(F)}\\
&=y\binom{k}{\ell}x^k\bar{z}^{\ell-k}\Lambda_{S'}(x+1,y,z,\bar{x}+1,\bar{y},\bar{z}).
\end{align*}

Meanwhile, when $k=\ell$, $X$ must consist of the $\ell$ odd elements of $S'\setminus S''$ along with a \emph{proper} subset of the maxima of $F$ (we considered the case when $X$ contained all of $F^{-1}(2m)\setminus\{2m\}$ already). This yields
\begin{align*}
\sum_{F \in \mathcal{E}_{S'}}&\sum_{\substack{X \subsetneq F^{-1}(2m)\setminus\{2m\}\\ |X\cap (S'\setminus S'')| = k}}w(\hat{F}_X)\\ 
&= \sum_{F \in \mathcal{E}_{S'}}\sum_{\substack{X \subsetneq F^{-1}(2m)\setminus\{2m\}\\ |X\cap (S'\setminus S'')| = \ell}}x^{\ell}y\left(x^{O(X)-\ell}y^{\text{fd}(F)}z^{\text{si}(F)}\bar{x}^{E(X)}\bar{y}^{\text{fi}(F)}\bar{z}^{\text{sd}(F)}\right)\\
&= y\sum_{F \in \mathcal{E}_{S'}}y^{\text{fd}(F)}z^{\text{si}(F)}\bar{y}^{\text{fi}(F)}\bar{z}^{\text{sd}(F)}x^{\ell}\sum_{A\times B \subsetneq \text{MO}(F)\times\text{ME}(F)}x^{|A|}\bar{x}^{|B|}\\
&= yx^{\ell}\sum_{F \in \mathcal{E}_{S'}}y^{\text{fd}(F)}z^{\text{si}(F)}\bar{y}^{\text{fi}(F)}\bar{z}^{\text{sd}(F)}\left((x+1)^{\text{mo}(F)}(\bar{x}+1)^{\text{me}(F)} - x^{\text{mo}(F)}\bar{x}^{\text{me}(F)}\right)\\
&= yx^{\ell}\left(\Lambda_{S'}(x+1,y,z,\bar{x}+1,\bar{y},\bar{z}) - \Lambda_{S'}(x,y,z,\bar{x},\bar{y},\bar{z})\right).
\end{align*}

Summing over all $k$ between $0$ and $\ell$ (including the case when $X = F^{-1}(2m)\setminus\{2m\})$, we obtain a total contribution of
\begin{align*}
y&\left(\sum_{k=0}^{\ell}\binom{\ell}{k}x^k\bar{z}^{\ell-k}\right)\Lambda_{S'}(x+1,y,z,\bar{x}+1,\bar{y},\bar{z}) - x^{\ell}(\bar{y}-y)\Lambda_{S'}(x,y,z,\bar{x},\bar{y},\bar{z})\\
&= y(x+\bar{z})^{\ell}\Lambda_{S'}(x+1,y,z,\bar{x}+1,\bar{y},\bar{z}) + x^{\ell}(\bar{y}-y)\Lambda_{S'}(x,y,z,\bar{x},\bar{y},\bar{z})   
\end{align*}
to $\Lambda_S$.

\textbf{Case 2.} Suppose that $2m \in X$.

Again, suppose that $X$ contains exactly $k$ of the $\ell$ odd elements of $T\setminus U$ for $0\leq k \leq \ell$.

When $k = \ell$, $X$ consists of $S'\setminus S''$ (which consists of $\ell$ odd elements and one even element) along with some proper subset of the maxima of $F$. In this case,
\begin{align*}w(\hat{F}_X) &= x^{O(X)}y^{\text{fd}(F)}z^{\text{si}(F)}\bar{x}^{E(X)}\bar{y}^{\text{fi}(F)}\bar{z}^{\text{sd}(F)}\\
&= x^{\ell}\bar{x}\left(x^{|A|}y^{\text{fd}(F)}z^{\text{si}(F)}\bar{x}^{|B|}\bar{y}^{\text{fi}(F)}\bar{z}^{\text{sd}(F)}\right),
\end{align*}
where $A$ and $B$ are subsets of $\text{MO}(F)$ and $\text{ME}(F)$, respectively, such that $A\times B \neq \text{MO}(F)\times\text{ME}(F)$.

Hence,
\begin{align*}
\sum_{F \in \mathcal{E}_{S'}}&\sum_{\substack{X \subseteq F^{-1}(2m)\\(S'\setminus S'')\subseteq X}}w(\hat{F}_X)\\ 
&= \sum_{F\in\mathcal{E}_{S'}}x^{\ell}\bar{x}\sum_{A\times B \subsetneq \text{MO}(F)\times\text{ME}(F)}x^{|A|}y^{\text{fd}(F)}z^{\text{si}(F)}\bar{x}^{|B|}\bar{y}^{\text{fi}(F)}\bar{z}^{\text{sd}(F)}\\
&= x^{\ell}\bar{x}\sum_{F \in \mathcal{E}_{S'}}y^{\text{fd}(F)}z^{\text{si}(F)}\bar{y}^{\text{fi}(F)}\bar{z}^{\text{sd}(F)}\sum_{A\times B \subsetneq \text{MO}(F)\times\text{ME}(F)}x^{|A|}\bar{x}^{|B|}\\
&=  x^{\ell}\bar{x}\sum_{F \in \mathcal{E}_{S'}}y^{\text{fd}(F)}z^{\text{si}(F)}\bar{y}^{\text{fi}(F)}\bar{z}^{\text{sd}(F)}\left((x+1)^{\text{mo}(F)}(\bar{x}+1)^{\text{me}(F)}-x^{\text{mo}(F)}\bar{x}^{\text{me}(F)}\right)\\
&= x^{\ell}\bar{x}\left(\Lambda_{S'}(x+1,y,z,\bar{x}+1,\bar{y},\bar{z}) - \Lambda_{S'}(x,y,z,\bar{x},\bar{y},\bar{z})\right).
\end{align*}

When $k = \ell - 1$, there are two possibilities. If $X = F^{-1}(2m)\setminus\{2i-1\}$ for some $2i-1\in S'\setminus S''$, then
\begin{align*}w(\hat{F}_X) &= x^{\text{mo}(F)+(\ell-1)}y^{\text{fd}(F)}z^{\text{si}(F)+1}\bar{x}^{\text{me}(F)+1}\bar{y}^{\text{fi}(F)}\bar{z}^{\text{sd}(F)}\\
&= x^{\ell-1}z\bar{x}w(F).\end{align*}
This can happen in $\ell$ ways, so such $\hat{F}_X$ give a contribution of $\ell x^{\ell-1}z\bar{x}\Lambda_{S'}$ to $\Lambda_S$.

On the other hand, if there is at least one maximum of $F$ that is not in $X$, then $X$ consists of $2m$, $(\ell-1)$ odd elements of $S'\setminus S''$, and some proper subset of the maxima of $F$. In this case,
\begin{align*}
w(\hat{F}_X) &= x^{|A|+(\ell-1)}y^{\text{fd}(F)}z^{\text{si}(F)}\bar{x}^{|B|+1}\bar{y}^{\text{si}(F)}\bar{z}^{\text{sd}(F)+1}\\
&= x^{\ell-1}\bar{x}\bar{z}\left(x^{|A|}y^{\text{fd}(F)}z^{\text{si}(F)}\bar{x}^{|B|}\bar{y}^{\text{si}(F)}\bar{z}^{\text{sd}(F)}\right),
\end{align*}
where $A$ and $B$ are subsets of the odd and even maxima of $F$, respectively, such that $A\times B \neq \text{MO}(F)\times \text{ME}(F)$. Note that for any such pair of $A$ and $B$, there are exactly $\ell$ such sets $X$.

Thus,
\begin{align*}
\sum_{F\in\mathcal{E}_{S'}}&\sum_{\substack{X\subsetneq F^{-1}(2m)\\ 2m \in X\\ X\cap (S'\setminus S'') = \ell}}w(\hat{F}_X)\\ 
&= \ell x^{\ell-1}\bar{x}\bar{z}\sum_{F\in\mathcal{E}_{S'}}y^{\text{fd}(F)}z^{\text{si}(F)}\bar{y}^{\text{si}(F)}\bar{z}^{\text{sd}(F)}\sum_{A\times B\subsetneq \text{MO}(F)\times\text{ME}(F)}x^{|A|}\bar{x}^{|B|}\\
&= \ell x^{\ell-1}\bar{x}\bar{z}\sum_{F\in\mathcal{E}_{S'}}y^{\text{fd}(F)}z^{\text{si}(F)}\bar{y}^{\text{si}(F)}\bar{z}^{\text{sd}(F)}\left((x+1)^{\text{mo}(F)}(\bar{x}+1)^{\text{me}(F)} - x^{\text{mo}(F)}\bar{x}^{\text{me}(F)}\right)\\
&= \ell x^{\ell-1}\bar{x}\bar{z}\left(\Lambda_{S'}(x+1,y,z,\bar{x}+1,\bar{y},\bar{z}) - \Lambda_{S'}(x,y,z,\bar{x},\bar{y},\bar{z})\right).
\end{align*}

When $0\leq k \leq \ell-2$ $X$ consists of: $2m$, $k$ odd elements of $S'\setminus S''$, and some subset (not necessarily proper) of the maxima of $F$. Hence,
\begin{align*}
w(\hat{F}_X) &= x^{|A|+k}y^{\text{fd}(F)}z^{\text{si}(F)}\bar{x}^{|B|+1}\bar{y}^{\text{fi}(F)}\bar{z}^{\text{sd}(F)+(\ell-k)}\\
&= x^k\bar{x}\bar{z}^{\ell-k}\left(x^{|A|}y^{\text{fd}(F)}z^{\text{si}(F)}\bar{x}^{|B|}\bar{y}^{\text{fi}(F)}\bar{z}^{\text{sd}(F)}\right),
\end{align*}
where $A$ and $B$ are subsets of the odd and even maxima of $F$, respectively. As before, given any pair $A$ and $B$, there are $\binom{\ell}{k}$ such sets $X$.

Therefore,
\begin{align*}
\sum_{F\in\mathcal{E}_{S'}}&\sum_{\substack{X\subsetneq F^{-1}(2m)\\ 2m \in X\\ X \cap (S'\setminus S'') = k+1}}w(\hat{F}_X) \\
&= \binom{\ell}{k}x^k\bar{x}\bar{z}^{\ell-k}\sum_{F\in\mathcal{E}_{S'}}y^{\text{fd}(F)}z^{\text{si}(F)}\bar{y}^{\text{fi}(F)}\bar{z}^{\text{sd}(F)}\sum_{A\times B\subseteq \text{MO}(F)\times \text{ME}(F)}x^{|A|}\bar{x}^{|B|}\\
&= \binom{\ell}{k}x^k\bar{x}\bar{z}^{\ell-k}\sum_{F \in \mathcal{E}_{S'}}y^{\text{fd}(F)}z^{\text{si}(F)}\bar{y}^{\text{fi}(F)}\bar{z}^{\text{sd}(F)}(x+1)^{\text{mo}(F)}(\bar{x}+1)^{\text{me}(F)}\\
&= \binom{\ell}{k}x^k\bar{x}\bar{z}^{\ell-k}\Lambda_{S'}(x+1,y,z,\bar{x}+1,\bar{y},\bar{z}).
\end{align*}

Summing over all $0\leq k \leq \ell$, this gives a total contribution of
\begin{align*}
\bar{x}&\left(\sum_{k=0}^{\ell}\binom{\ell}{k}x^k\bar{z}^{\ell-k}\right)\Lambda_{S'}(x+1,y,z,\bar{x}+1,\bar{y},\bar{z}) + \bar{x}(\ell x^{\ell-1}z - x^{\ell} - \ell x^{\ell-1}\bar{z})\Lambda_{S'}(x,y,z,\bar{x},\bar{y},\bar{z})\\
&= \bar{x}(x+\bar{z})^{\ell}\Lambda_{S'}(x+1,y,z,\bar{x}+1,\bar{y},\bar{z}) + x^{\ell-1}(\ell \bar{x}(z - \bar{z}) - x\bar{x})\Lambda_{S'}(x,y,z,\bar{x},\bar{y},\bar{z})
\end{align*}
to $\Lambda_S$.

Finally, we combine the contributions from Cases 1 and 2 to obtain:
\begin{align*}
\Lambda_S(x,y,z,\bar{x},\bar{y},\bar{z}) = &(y+\bar{x})(x+\bar{z})^{\ell}\Lambda_{S'}(x+1,y,z,\bar{x}+1,\bar{y},\bar{z})\\ &+ x^{\ell-1}[x(\bar{y}-y) + \ell \bar{x}(z-\bar{z}) - x\bar{x}]\Lambda_{S'}(x,y,z,\bar{x},\bar{y},\bar{z}).
\end{align*}
\end{proof}

\subsection{D-permutations and Generalized Surjective Staircases}

The following lemma was proved in \cite[Lemma 5.2]{Type_A_Paper}:

\begin{lemma}\label{bijectionProp}
	There is a bijection $$\phi:  \mathcal{D}_{2n} \to  \{ f \in \mathcal{E}_{2n+2}: f \mbox { has  no even maxima} \}$$ such that for all $\sigma \in \mathcal{D}_{2n}$ and $j \in [2n]$, the following properties hold:
	\begin{enumerate}
		\item $j$ is an even cycle maximum of $\sigma$ if and only if it is a fixed point of $\phi(\sigma)$,
		\item $j$ is an even fixed point of $\sigma$ if and only if it is an isolated fixed point of $\phi(\sigma)$,
		\item  $j$ is an odd fixed point of $\sigma$ if and only if it is an odd maximum of $\phi(\sigma)$.
	\end{enumerate} 
\end{lemma}

Our next goal is to generalize Lemma \ref{bijectionProp}. Let $V$ be a finite set of positive integers whose maximum is even, and let $R_k$ be a finite set of $k+1$ positive integers consisting of:
\begin{itemize}
\item $k$ odd numbers larger than $\max V$, and
\item one even number that is larger than the odd numbers.
\end{itemize}

We will show that the D-permutations on $V$ are in bijection with certain of the surjective staircases with domain $V\cup R_k$.

\begin{lemma}\label{GeneralizedbijectionProp}
For any $k$, let $V$ and $R_k$ be as above. There is a bijection
$$\tilde{\phi}: \mathcal{D}_{V} \to \{f \in \mathcal{E}_{V\cup R_k} \st f \text{ has no even maxima}\},$$
such that for all $\sigma \in \mathcal{D}_{V}$ and all $j \in V$
\begin{enumerate}
\item $j$ is an even cycle maximum of $\sigma$ if and only if $j$ is a fixed point of $\phi(\sigma)$,
\item $j$ is an even fixed point of $\sigma$ if and only if $j$ is an isolated fixed point of $\phi(\sigma)$,
\item $j$ is an odd fixed point of $\sigma$ if and only if $j$ is an odd maximum of $\phi(\sigma)$.
\end{enumerate}
\end{lemma}

\begin{proof}
Let $\sigma \in \mathcal{D}_V$. We wish to apply Lemma \ref{bijectionProp}. Let $2m = \max V$, $2n = \max R_k$ and let $\sigma'$ be the permutation in $\mathfrak{S}_{2n-2}$ defined by
$$\sigma'(i) = \begin{cases}\sigma(i), & i \in V\\ i, & \text{otherwise}\end{cases}.$$
Notice that the map $\psi:\{\sigma \in \mathfrak{S}_V\}\to\{\sigma \in \mathfrak{S}_{2n-2} \st \sigma(i) = i \Forall i \notin V\}$ given by $\sigma \mapsto \sigma'$ is a bijection.

We apply the map $\phi$ from Lemma \ref{bijectionProp} to $\sigma'$ to obtain a surjective staircase $f' \in \mathcal{E}_{2n}$ with no even maxima. Our next goal is to show that the restriction of $f'$ to $V\cup R_k$ is a surjective staircase.

First, recall that for all $i \in [2n-2]\setminus V$, $\sigma'(i) = i$. Hence, if $i \in [2n-2]\setminus V$ is even, $i$ is an isolated fixed point of $f'$, and if $i \in [2n-2]\setminus V$ is odd, $i$ is an odd maximum of $f'$. This tells us immediately that for all $i \in V$, $f'(i) \in [2n]\setminus ([2n-2]\setminus V) = V \cup \{2n-1,2n\}$. Since the image of $f'$ contains no odd numbers, we see that $f'(i) \in V \cup \{2n\}$ for all $i \in V$. In other words,
$$f'(V) \subseteq \{\text{even elements of }V\} \cup \{2n\}.$$

Conversely, we claim that $\{\text{even elements of }V\} \subseteq f'(V)$. Indeed, every element of $[2n-2]\setminus V$ is either a fixed point or a maximum of $f'$, so if $i \in [2n-2]\setminus V$, $f'(i)\notin V$. By definition, $f'(2n-1) = f'(2n) = 2n$. Hence, $\{\text{even elements of }V\} \nsubseteq f'([2n]\setminus V)$, so the only way for $f'$ to be surjective is to have 
$$\{\text{even elements of }V\} \subseteq f'(V).$$

We also see that $f'(R_k) = \{2n\}$. Indeed, we know that $f'(2n-1) = f'(2n) = 2n$ (whether or not $2n-1 \in R_k$). Any other element $i$ of $R_k$ is an odd element of $[2n-2]\setminus V$, and by construction, all such $i$ are odd maxima of $f'$. Hence,
$$f'(V\cup R_k) = \{\text{even elements of }V\} \cup \{2n\}.$$ 

Thus, we can define a surjective staircase $f \in \mathcal{E}_{V\cup R_k}$ by restricting the domain of $f'$ to $V \cup R_k$. Indeed, we have just shown that the image of such an $f$ would be $\{\text{even elements of V }\} \cup \{2n\}$, and such an $f$ would be excedent because $f'$ was. We also know that $f$ would have no even maxima because $f'$ had no even maxima.

Finally, note that the image of $\phi \circ \psi$ consists of surjective staircases $f' \in \mathcal{E}_{2n}$ with no even maxima, such that $f'(2i-1) = 2n$ for all $2i-1\notin V$, and such that $2i$ is an isolated fixed point of $f'$ for all $2i \notin V$. We observe that restricting the domain of $f'$ to $V \cup R_k$ yields a bijection between the image of $\phi\circ\psi$ and $\{f \in \mathcal{E}_{V \cup R_k} \st f \text{ has no even maxima}\}$. Indeed, given any $f \in \mathcal{E}_{V\cup R_k}$ that has no even maxima, we define a $g \in \mathcal{E}_{2n}$ with no even maxima by
$$g(a) = \begin{cases}
f(a), & a \in V\\
2n, & a\in [2n]\setminus V \text{ is odd}\\
a, & a\in [2n]\setminus V \text{ is even}
\end{cases}.$$
Such a $g$ is in the image of $\phi\circ\psi$ because $2i-1$ is an odd maximum of $g$ for all $2i-1 \in [2n-2]\setminus V$ and $2i$ is an isolated fixed point of $g$ for all $2i \in [2n-2]\setminus V$. Moreover, by construction we see that $g|_{V\cup R_k} = f$.

Hence, we define $\tilde{\phi}: \mathcal{D}_V \to \mathcal{E}_{V \cup R_k}$ by $\tilde{\phi}(\sigma) = \left.\phi(\sigma')\right|_{V \cup R_k}$. Since the three desired properties held for $\phi$, they hold for $\tilde{\phi}$ automatically. The map $\tilde{\phi}$ is the composition of three maps $\sigma \mapsto \sigma' \mapsto f' \mapsto f$, each of which is a bijection, so $\tilde{\phi}$ is a bijection.
\end{proof}

Next, we prove a technical result generalizing \cite[Lemma 5.4]{Type_A_Paper}.

\begin{lemma}\label{FerrersDPermLemma}
For any $k$,
$$\sum_{\sigma \in \mathcal{D}_{V}}t^{c(\sigma)} = \Lambda_{V\cup R_k}(t,t,1,0,t,1),$$
where $c(\sigma)$ is the number of cycles of $\sigma$.
\end{lemma}

\begin{proof} We apply the bijection $\tilde{\phi}$ from Lemma \ref{GeneralizedbijectionProp}:
\begin{align*}
\sum_{\sigma \in \mathcal{D}_{V}}t^{c(\sigma)} &= \sum_{\sigma \in \mathcal{D}_{V}} t^{\text{fi}(\tilde{\phi}(\sigma)) + \text{fd}(\tilde{\phi}(\sigma))+\text{mo}(\tilde{\phi}(\sigma))}\\
&= \sum_{\substack{f \in \mathcal{E}_{V\cup R_k}\\ \text{me}(f) = 0}}t^{\text{fi}(f) + \text{fd}(f) + \text{mo}(f)}\\
&= \Lambda_{V\cup R_k}(t,t,1,0,t,1).
\end{align*} \end{proof}

\section{Generating Function Proofs}\label{GenFunProofSec}

We now use the technical results established in Section \ref{SurjStaircaseTechniques} to prove Theorems \ref{kStaircaseCharGenTh} and \ref{CompBipGenThm}.

\subsection{\texorpdfstring{$k$}{k}-step staircases}
Let 
$$S_n^k = \{1,3,5,\dots,2nk-1\} \sqcup \{2k,4k,\dots,2nk\},$$
and note that $\lambda(S_n^k)$ is the $k$-staircase partition $\lambda^{(n)}_k$.

We are ready to prove a generalization of Randrianarivony and Zeng's generating function formula for the generalized Dumont--Foata polynomials.
\begin{theorem}\label{FixedStepGenFun}
For all $k\geq 1$,
\begin{align*}
\sum_{n\geq 1}&\Lambda_{S_n^k}(x,y,z,\bar{x},\bar{y},\bar{z})u^{n-1}\\
&= \sum_{n\geq 1}\frac{(y+\bar{x})^{(n-1)}\left((x+\bar{z})^{(n-1)}\right)^{k}u^{n-1}}{\prod_{i=0}^{n-1}\left(1-(x+i)^{k-1}\left[(x+i)(\bar{y}-y) + k(x+i)(z-\bar{z}) - (x+i)(\bar{x}+i)\right]u\right)},
\end{align*}
where $x^{(n)} = x(x+1)\cdots(x+(n-1))$.
\end{theorem}

\begin{proof}
Throughout this proof we will abbreviate $\Lambda_{S}(x,y,z,\bar{x},\bar{y},\bar{z})$ as $\Lambda_{S}$, and similarly write $\Lambda_S(x+1,y,z,\bar{x}+1,\bar{y},\bar{z})$ as $\Lambda_{S}(x+1,\bar{x}+1)$. By assumption, $S_1^k$ contains exactly one even element, so $\Lambda_{S_1^k} = 1$. We can thus write
\begin{align*}\sum_{n\geq 1}\Lambda_{S_n^k}u^{n-1} &= 1 + \sum_{n\geq 2}\Lambda_{S_n^k}u^{n-1}\\
&= 1 + \sum_{n\geq 2}(y+\bar{x})(x+\bar{z})^{k}\Lambda_{S_{n-1}^k}(x+1,\bar{x}+1)u^{n-1}\\ &\;\;\;\;+ \sum_{n\geq 2}x^{k-1}[x(\bar{y}-y) + k \bar{x}(z-\bar{z}) - x\bar{x}]\Lambda_{S_{n-1}^k}u^{n-1}\\
&= 1 + \sum_{n\geq 1}(y+\bar{x})(x+\bar{z})^{k}u\Lambda_{S_{n}^k}(x+1,\bar{x}+1)u^{n-1}\\ &\;\;\;\;+ \sum_{n\geq 1}x^{k-1}[x(\bar{y}-y) + k \bar{x}(z-\bar{z}) - x\bar{x}]u\Lambda_{S_{n}^k}u^{n-1},\end{align*}
with the second equality being the recurrence proved in Theorem \ref{GeneralizedRecursion}.

We can rearrange this equation to obtain
\begin{align*}\left(1-x^{k-1}[x(\bar{y}-y) + k \bar{x}(z-\bar{z}) - x\bar{x}]u\right)&\sum_{n\geq 1}\Lambda_{S_n^k}u^{n-1}\\ 
&= 1+ ((y+\bar{x})(x+\bar{z})^{k}u)\sum_{n\geq 1}\Lambda_{S_n^k}(x+1,\bar{x}+1)u^{n-1},
\end{align*}
so we have
\begin{align*}
\sum_{n\geq 1}\Lambda_{S_n^k}u^{n-1} &= \frac{1}{1-x^{k-1}[x(\bar{y}-y) + k \bar{x}(z-\bar{z}) - x\bar{x}]u}\\
& + \frac{(y+\bar{x})(x+\bar{z})^{k}u}{1-x^{k-1}[x(\bar{y}-y) + k \bar{x}(z-\bar{z}) - x\bar{x}]u}\sum_{n\geq 1}\Lambda_{S_n^k}(x+1,\bar{x}+1)u^{n-1}.
\end{align*}

By recursively expanding the right-hand side of this equation, we obtain the desired generating function formula.
\end{proof}

We can now prove Theorem \ref{kStaircaseCharGenTh}.

\begin{proof}[Proof of Theorem \ref{kStaircaseCharGenTh}.]
By Corollary \ref{LambdaGammaThm}, we know that $\Gamma_{S_n^k} = \Gamma_{V_n^k}$ for any $V_n^k$ satisfying $\lambda(V_n^k) = \lambda^{(n)}_k$, so without loss of generality we can take $V_n^k = S_n^k$. By Lemma \ref{FerrersDPermLemma}, we know that 
$$\sum_{\sigma \in \mathcal{D}_{S_n^k}}t^{c(\sigma)} = \Lambda_{S_{n+1}^k}(t,t,1,0,t,1).$$ 
Hence, we have
\begin{align*}\sum_{n\geq 1}\sum_{\sigma \in \mathcal{D}_{S_n^k}}t^{c(\sigma)}u^n &= \sum_{n\geq 1}\Lambda_{S_{n+1}^k}(t,t,1,0,t,1)u^n\\
 &= -1 +\sum_{n\geq 1}\Lambda_{S_n^k}(t,t,1,0,t,1)u^{n-1}\\
&= -1 + \sum_{n\geq 1} \frac{t^{(n-1)}\left((t+1)^{(n-1)}\right)^ku^{n-1}}{\prod_{i=0}^{n-1}\left(1+i(t+i)^{k}u\right)}\\
&= \sum_{n\geq 2} \frac{t^{(n-1)}\left((t+1)^{(n-1)}\right)^ku^{n-1}}{\prod_{i=0}^{n-1}\left(1+i(t+i)^{k}u\right)}\\
&= \sum_{n\geq 1} \frac{t^{(n)}\left((t+1)^{(n)}\right)^ku^{n}}{\prod_{i=0}^{n}\left(1+i(t+i)^{k}u\right)}, \end{align*}
with the third equality following from Theorem \ref{FixedStepGenFun}.

Next, we have 
\begin{align}\sum_{n\geq 1}\sum_{\sigma \in \mathcal{D}_{S_n^k}}(-t)^{c(\sigma) - 1}u^n &= -\frac{1}{t}\sum_{n\geq 1} \frac{(-t)^{(n)}\left((-t+1)^{(n)}\right)^ku^{n}}{\prod_{i=0}^{n}\left(1+i(-t+i)^{k}u\right)}\nonumber\\
&= \sum_{n\geq 1} \frac{(-1)^{nk + n-1}(t-1)_{n-1}\left((t-1)_{n}\right)^ku^{n}}{\prod_{i=0}^{n}\left(1+i(i-t)^{k}u\right)}\label{kStepEq}.\end{align}

By Theorem \ref{mobth}, we know that $\chi_{\Pi_{\Gamma_{S_n^k}}}(t) = (-1)^{nk + n - 1}\displaystyle\sum_{\sigma \in \mathcal{D}_{S_n^k}}(-t)^{c(\sigma) - 1}$, because $\Gamma_{S_n^k}$ has $nk + n$ vertices. Hence,
\begin{align*}
\sum_{n\geq 1}\chi_{\Pi_{\Gamma_{S_n^k}}}(t)u^n &= -\sum_{n\geq 1}\sum_{\sigma \in \mathcal{D}_{S_n^k}}(-t)^{c(\sigma) - 1}(-1)^{nk + n}u^n\\
&= -\sum_{n\geq 1}\sum_{\sigma \in \mathcal{D}_{S_n^k}}(-t)^{c(\sigma) - 1}((-1)^{k+1}u)^n\\
&= -\sum_{n\geq 1} \frac{(-1)^{nk + n-1}(t-1)_{n-1}\left((t-1)_{n}\right)^k(-1)^{nk + n}u^{n}}{\prod_{i=0}^{n}\left(1+i(i-t)^{k}(-1)^{k+1}u\right)}\\
&= \sum_{n\geq 1} \frac{(t-1)_{n-1}\left((t-1)_{n}\right)^ku^{n}}{\prod_{i=0}^{n}\left(1-i(t-i)^{k}u\right)},
\end{align*}
with the third equality following from Equation (\ref{kStepEq}).
\end{proof}

\subsection{The complete bipartite graph \texorpdfstring{$K_{n,k}$}{K(n,k)}.}
Fix $k \geq 1$, and let $T_n^k = \{1,3,5,\dots,2k-1\}\sqcup\{2k,\dots,2(k+n-1)\}$ as in Section \ref{CompBipSection}. By applying Theorem \ref{GeneralizedRecursion}, we see that 
$$\Lambda_{T_1^k} = 1,$$ 
\begin{align}\label{CompBipBase}\Lambda_{T_2^k} &= (y+\bar{x})(x+\bar{z})^k + x^{k-1}[x(\bar{y} - y) - k\bar{x}(\bar{z}-z) - x\bar{x}],\end{align}
and
\begin{align}\label{CompBipRec}
\Lambda_{T_n^k} &= (y+\bar{x})\Lambda_{T_{n-1}}(x+1,\bar{x}+1) + (\bar{y} - y -\bar{x})\Lambda_{T_{n-1}^k}
\end{align}
for all $n\geq 3$, where $\Lambda_{T_n^k}(x+1,\bar{x}+1) \coloneqq \Lambda_{T_n^k}(x+1,y,z,\bar{x}+1,\bar{y},\bar{z})$.

\begin{theorem}\label{NoStepGenEq}
We have
$$\sum_{n\geq 1}\Lambda_{T_n^k}u^{n-1} = 1 + \sum_{n\geq 2}\frac{(y+\bar{x})^{(n-2)}\Lambda_{T_2^k}(x+(n-2),\bar{x}+(n-2))u^{n-1}}{\prod_{i=0}^{n-2}(1-(\bar{y} - y - (\bar{x}+i))u)},$$
where $a^{(n)}$ is the rising factorial $a(a+1)\cdots (a+(n-1))$.
\end{theorem}

\begin{proof}
We can write
\begin{align*}\sum_{n\geq 1}&\Lambda_{T_n^k}u^{n-1}\\ 
&= 1 + \Lambda_{T_2^k}u + \sum_{n\geq 3}\Lambda_{T_n^k}u^{n-1}\\
&= 1 + \Lambda_{T_2^k}u + (y+\bar{x})\sum_{n\geq 3}\Lambda_{T_{n-1}^k}(x+1,\bar{x}+1)u^{n-1} + (\bar{y}- y -\bar{x})\sum_{n\geq 3}\Lambda_{T_{n-1}^k}u^{n-1}\\
&= 1 + \Lambda_{T_2}u + (y+\bar{x})u\sum_{n\geq 2}\Lambda_{T_{n}^k}(x+1,\bar{x}+1)u^{n-1} +(\bar{y}-y -\bar{x})u\sum_{n\geq 2}\Lambda_{T_n^k}u^{n-1},
\end{align*}
with the second equality following from Equation (\ref{CompBipRec}).

Subtracting $1$ from both sides, we have
$$\sum_{n\geq 2}\Lambda_{T_n^k}u^{n-1} = \Lambda_{T_2^k}u + (y+\bar{x})u\sum_{n\geq 2}\Lambda_{T_{n}^k}(x+1,\bar{x}+1)u^{n-1} +(\bar{y}-y -\bar{x})u\sum_{n\geq 2}\Lambda_{T_n^k}u^{n-1}.$$

By rearranging, we see
$$(1-(\bar{y}-y-\bar{x})u)\sum_{n\geq 2}\Lambda_{T_n^k}u^{n-1} = \Lambda_{T_2^k}u + (y+\bar{x})u\sum_{n\geq 2}\Lambda_{T_{n}^k}(x+1,\bar{x}+1)u^{n-1},$$
or equivalently that
$$\sum_{n\geq 2}\Lambda_{T_n^k}u^{n-1} = \frac{\Lambda_{T_2^k}u}{1 - (\bar{y}-y-\bar{x})u} + \frac{(y+\bar{x})u}{1-(\bar{y} - y - \bar{x})u}\sum_{n\geq 2}\Lambda_{T_{n}^k}(x+1,\bar{x}+1)u^{n-1}.$$

By recursively expanding the right-hand side of the above equation, we see that
$$\sum_{n\geq 2}\Lambda_{T_n^k}u^{n-1} = \sum_{n\geq 2}\frac{(y+\bar{x})^{(n-2)}\Lambda_{T_2^k}(x+(n-2),\bar{x}+(n-2))u^{n-1}}{\prod_{i=0}^{n-2}(1-(\bar{y} - y - (\bar{x}+i))u)}.$$
\end{proof}

Now, by applying Lemma \ref{FerrersDPermLemma}, we can prove Theorem \ref{CompBipGenThm}.

\begin{proof}[Proof of Theorem \ref{CompBipGenThm}.]
By Lemma \ref{FerrersDPermLemma}, we know that
$$\sum_{\sigma \in \mathcal{D}_{T_i^k}}t^{c(\sigma)} = \Lambda_{T_{i+1}^k}(t,t,1,0,t,1),$$
so
\begin{align*}
\sum_{n\geq 1}\sum_{\sigma \in \mathcal{D}_{T_n^k}}t^{c(\sigma)}u^n &= \sum_{n\geq 2}\Lambda_{T_n^k}(t,t,1,0,t,1)u^{n-1}\\
&= \sum_{n\geq 2}\frac{t^{(n-2)}\left.\left(\Lambda_{T_2^k}(x+(n-2),\bar{x}+(n-2))\right)\right|_{(t,t,1,0,t,1)}u^{n-1}}{\prod_{i=0}^{n-2}(1+iu)},
\end{align*}
by Theorem \ref{NoStepGenEq}.

Now, by Equation (\ref{CompBipBase}),
\begin{align*}\Lambda_{T_2^k}(x+i,\bar{x}+i) = &(y+\bar{x}+i)(x+\bar{z}+i)^k\\ &+ (x+i)^{k-1}\left[(x+i)(\bar{y}-y) -k(\bar{x}+i)(\bar{z}-z) - (x+i)(\bar{x}+i)\right],\end{align*}
so
\begin{align*}\left.\left(\Lambda_{T_2^k}(x+i,\bar{x}+i)\right)\right|_{(t,t,1,0,t,1)} &= (t+i)(t+i+1)^k + (t+i)^{k-1}(-(t+i)i)\\
&= (t+i)(t+i+1)^k -i(t+i)^k.
\end{align*}

Thus, after reindexing we have
\begin{align*}\sum_{n\geq 1}\sum_{\sigma \in \mathcal{D}_{T_n^k}}t^{c(\sigma)}u^n &= \sum_{n\geq 1}\frac{t^{(n-1)}((t+(n-1))(t+n)^k - (n-1)(t+(n-1))^k)u^{n}}{\prod_{i=1}^{n-1}(1+iu)}\\
&= t(t+1)^ku + \sum_{n\geq 2}\frac{t^{(n)}\left[(t+n)^k - (n-1)(t+(n-1))^{k-1}\right]u^n}{\prod_{i=0}^{n-1}(1+iu)}\end{align*}

Next, we have
\begin{align*}\sum_{n\geq 1}&\sum_{\sigma \in \mathcal{D}_{T_n^k}}(-t)^{c(\sigma) - 1}\\ 
&= -\frac{1}{t}\left((-t)(1-t)^ku + \sum_{n\geq 2}\frac{(-t)^{(n)}\left[(n-t)^k - (n-1)((n-1)-t)^{k-1}\right]u^n}{\prod_{i=0}^{n-1}(1+iu)}\right)\\
&= (1-t)^ku + \sum_{n\geq 2}\frac{(-1)^{n +k -1}(t-1)_{n-1}\left[(t-n)^k + (n-1)(t-(n-1))^{k-1}\right]u^n}{\prod_{i=0}^{n-1}(1+iu)}.
\end{align*}

Finally, since $\Gamma_{T_n^k}$ has $n+k$ vertices, $\chi_{\Pi_{\Gamma_{T_n^k}}}(t) = (-1)^{n+k-1}\displaystyle\sum_{\sigma \in \mathcal{D}_{T_n^k}}(-t)^{c(\sigma)-1}$ by Theorem \ref{mobth}. Hence,
\begin{align*}&\sum_{n\geq 1}\chi_{\Pi_{\Gamma_{T_n^k}}}(t)u^n\\
&= \sum_{n\geq 1}(-1)^{n+k-1}\sum_{\sigma \in \mathcal{D}_{T_n^k}}(-t)^{c(\sigma)-1}u^n\\
&= (-1)^{k-1}\sum_{n\geq 1}\sum_{\sigma \in \mathcal{D}_{T_n}}(-t)^{c(\sigma)-1}(-u)^n\\
&= (t-1)^ku\\ &\hspace{5mm}+ (-1)^{k-1}\sum_{n\geq 2}\frac{(-1)^{n +k -1}(t-1)_{n-1}\left[(t-n)^k + (n-1)(t-(n-1))^{k-1}\right](-1)^nu^n}{\prod_{i=0}^{n-1}(1-iu)}\\
&= (t-1)^ku + \sum_{n\geq 2}\frac{(t-1)_{n-1}\left[(t-n)^k + (n-1)(t-(n-1))^{k-1}\right]u^n}{\prod_{i=0}^{n-1}(1-iu)}\\
&= \sum_{n\geq 1}\frac{(t-1)_{n-1}\left[(t-n)^k + (n-1)(t-(n-1))^{k-1}\right]u^n}{\prod_{i=0}^{n-1}(1-iu)}.
\end{align*}
Finally, we multiply by $t$ to obtain the desired generating function for the chromatic polynomial.
\end{proof}

\section{Dowling Analogs}\label{DowlingSec}

\subsection{Homogenized \texorpdfstring{$\nu$}{ν}-Dowling Arrangements}

As in \cite[Chapter 4]{Alex_Thesis} (see also \cite{LaWa}), we can consider a complex analog of the homogenized $\nu$-arrangements. Many of the results of the previous sections can be extended to the complex case after some technical results are established.

Let $\nu = (\nu_1,\dots,\nu_n)$ be a weak composition of $m$, and fix a positive integer $q$. Let $\omega$ be the primitive $q$th root of unity $e^{\frac{2 \pi i}{q}}$. We define the \emph{homogenized $\nu$-Dowling arrangement} to be the hyperplane arrangement in
$$\{(x_1,\dots,x_n, y_1^{(1)},\dots,y_1^{(\nu_1)},y_2^{(1)}.\dots,y_2^{(\nu_2)},\dots,y_n^{(1)},\dots,y_n^{(\nu_n)})\} = \C^{m+n}$$
given by
\begin{align*}\mathcal{H}_{\nu}^q = &\{x_i - \omega^{k}x_j = y_i^{(\ell)} \st 1 \leq i < j \leq n, 0\leq k \leq q-1, 1 \leq \ell \leq \nu_i\}\\ &\cup \{x_i = y_i^{(\ell)} \st 1 \leq i \leq n, 1\leq \ell \leq \nu_i\}.
\end{align*}

The term ``Dowling'' comes from the fact that the intersection lattice of $\mathcal{H}_{\nu}^q$ is a subposet of the \emph{Dowling lattice} $Q_n(\Z_q)$, introduced by Dowling in \cite{Dowling_Geometric_Lattices}.

Let $\mathcal{K}_{\nu}^q$ be the hyperplane arrangement in $\C^{n+m}$ given by:
\begin{align*}\{x_{2i-1}^{(\ell)}-\omega^kx_{2j-2} = 0 &\st 1\leq i <j \leq n, 0\leq k\leq q-1, 1\leq \ell \leq \nu_i\}\\ &\cup \{x_{2i-1}^{(\ell)} = 0 \st 1\leq i\leq n, 1\leq \ell \leq \nu_i\}.
\end{align*}

\begin{lemma}
There is an invertible linear transformation from $\C^{n+m}$ to itself that takes the hyperplanes of $\mathcal{H}_{\nu}^q$ to the hyperplanes of $\mathcal{K}_{\nu}^q$. In particular, we have $\mathcal{L}(\mathcal{H}_{\nu}^q) \cong \mathcal{L}(\mathcal{K}_{\nu}^q)$. 
\end{lemma}

\begin{proof}
This follows from essentially the same argument as in the proof of Theorem \ref{LambdaBondTh}. See also \cite[Lemma 4.2.3]{LaWa, Alex_Thesis}.
\end{proof}

The results of \cite[Section 4.2]{Alex_Thesis} (and of \cite{LaWa}) allow us to interpret the coefficients of the $\chi_{\mathcal{L}(\mathcal{K}_{\nu}^q)}(t)$ (and hence of $\chi_{\mathcal{L}(\mathcal{H}_{\nu}^q)}(t)$) in terms of certain decorated D-permutations.

A \emph{$q$-labeled D-cycle} on $V\subseteq [2r]$ is a D-cycle $\sigma$ on $V$, some of whose entries are labeled with elements of $\{0,\dots,q-1\}$ subject to the following conditions.
\begin{itemize}
	\item The maximum entry of $\sigma$ is labeled $0$.
	\item If $2r \in V$ and $\sigma$ is written in the form $(w\cdot 2r)$ then the right-to-left minima of the word $w$ are the only unlabeled entries of $\sigma$.
	\item If $2r \notin V$ then each entry of $\sigma$ is labeled.
\end{itemize} 

A \emph{$q$-labeled D-permutation} on $V \subseteq [2r]$ is a D-permutation $\sigma$ on $V$, some of whose entries are labeled with elements of $\{0,\dots,q-1\}$, such that each cycle of $\sigma$ is a $q$-labeled D-cycle. We write $\mathcal{D}_{V,r}^q$ for the set of $q$-labeled D-permutations on $V\subseteq [2r]$.

The combinatorics of $q$-labeled D-permutations is equivalent to the combinatorics of a certain class of edge-decorated forests. 

A tree $T$ on vertex set $V \subset \Z_{>0}$ is \emph{increasing-decreasing} (or \emph{ID} for short) if, when $T$ is rooted at its largest vertex, each internal vertex $v$ of $T$ satisfies:
\begin{itemize}
	\item if $v$ is odd then $v$ is smaller than all of its descendants and all of the children of $v$ are even,
	\item if $v$ is even then $v$ is larger than all of its descendants and all of the children of $v$ are odd.
\end{itemize}

An \emph{ID} forest is a forest, each of whose connected components is an ID tree. ID forests were introduced in \cite{Type_A_Paper}, where the authors showed that the ID forests on $V$ are exactly the non-broken circuit sets (with respect to a particular edge order) of the graph $\Gamma_V$ \cite[Theorem 3.5]{Type_A_Paper}.

A \emph{$q$-labeled ID tree} is a tree $T$ on vertex set $V \subseteq [2r]$ such that each edge of $T$ not incident to $2r$ is labeled with an element of $\{0,\dots,q-1\}$ (so if $2r \notin V$, each edge of $T$ is labeled). A \emph{$q$-labeled ID forest} is a forest, each of whose components is a $q$-labeled ID tree. We write $\mathcal{F}_{V,r}^q$ for the set of $q$-labeled ID forests on $V\subseteq [2r]$. As in the unlabeled case, the $q$-labeled ID forests are exactly the non-broken circuit sets of the matroid arising from a complex hyperplane arrangement studied by Lazar and Wachs \cite[Proposition 4.2.8]{LaWa,Alex_Thesis}.

In \cite[Theorem 4.2.10]{LaWa,Alex_Thesis}, Lazar and Wachs construct a bijection between the $q$-labeled D-permutations on $V$ with $k$ cycles and the $q$-labeled ID forests on $V$ with $k$ components.   

Let $V$ be any finite subset of $\Z_{>0}$ with odd minimum and even maximum such that $\lambda(V) = \lambda(\nu)$, and let $2r = \max V$. Then $V$ has $m$ odd elements and $n$ even elements. We let $2s$ be the second-largest even element of $V$ and suppose that there are $\ell$ odd elements of $V$ between $2s$ and $2r$. Finally, we let $V' \coloneqq V|_{[2s]}$. 

By \cite[Proposition 4.2.8]{LaWa,Alex_Thesis} and \cite[Theorem 4.2.10]{LaWa,Alex_Thesis}, we have the following interpretation of the coefficients of $\chi_{\mathcal{L}(\mathcal{K}_{\nu}^q)}(t)$.

\begin{proposition}
The coefficient of $(-1)^kt^{k-1}$ in $\chi_{\mathcal{L}(\mathcal{K}_{\nu}^q)}(t)$ is equal to the number of $q$-labeled D-permutations on $V$ with exactly $k$ cycles, and hence is equal to the number of $q$-labeled ID forests on $V$ with exactly $k$ components.
\end{proposition}

We can thus use Lemma \ref{GeneralizedbijectionProp} to show the following:

\begin{theorem}\label{DowlingDuFoTheorem}
For any weak composition $\nu$ of $m$ of length $n$ and any finite $V\subseteq \Z_{>0}$ with odd minimum and even maximum such that $\lambda(\nu) = \lambda(V)$, we have
$$\chi_{\mathcal{L}(\mathcal{H}_{\nu}^q)}(t) = (-1)^{m+n-1}(1-t)^{\ell}q^{m+n-\ell - 1}\Lambda_{V}\left(\frac{1-t}{q},\frac{1-t}{q},1,0,\frac{-t}{q},1\right),$$
where $\ell$ is the number of odd elements of $V$ between the largest and second-largest even elements of $V$.
\end{theorem}

\begin{proof}
We have 
$$\chi_{\mathcal{L}(\mathcal{H}_{\nu}^q)}(t) = \chi_{\mathcal{L}(\mathcal{K}_{\nu}^q)}(t) = (-1)^{m+n-1}\sum_{F \in \mathcal{F}_{V,r}^q}(-t)^{c(F)-1},$$
where $c(F)$ is the number of components of $F$. We let $F'$ be the forest obtained from $F$ by deleting the vertex $2r$. Let $G$ is a $q$-labeled ID forest on $V\setminus\{2r\}$ such that $F' = G$, and let $T$ be a component of $G$. If $T$ is not an isolated even vertex then either $T$ is a component of $F$, or $T$ is attached to $2r$ in $F$. If $T$ is an isolated even vertex then $T$ must also be a component of $F$. Hence
$$\sum_{\substack{F \in \mathcal{F}_{V,r}^q\\F'=G}}(-t)^{c(G)-1} = (-t)^{\#\{\text{even isolated nodes of }G\}}(1-t)^{\#\{\text{other components of }G\}}.$$

Moreover, we know that $G$ must have at least $\ell$ isolated odd nodes for the $\ell$ odd elements of $V$ between $2s$ and $2r$ (since each such vertex can only share an edge with $2r$). Hence we have

\begin{align*}
\sum_{F \in \mathcal{F}_{V,r}^q}(-t)^{c(F)-1} &= \sum_{G \in \mathcal{F}_{V\setminus\{2r\},r}^q}\sum_{\substack{F \in \mathcal{F}_{V,r}^q\\F'=G}}(-t)^{c(G)-1}\\ 
&= \sum_{G\in \mathcal{F}_{V\setminus\{2r\},n}^q}(-t)^{\#\{\text{even isolated nodes of }G\}}(1-t)^{\#\{\text{other components of }G\}}\\
&= (1-t)^{\ell}\sum_{G\in \mathcal{F}_{V',r}^q}(-t)^{\#\{\text{even isolated nodes of }G\}}(1-t)^{\#\{\text{other components of }G\}}.
\end{align*}

By applying the bijection between $\mathcal{F}_{V',r}^q$ and $\mathcal{D}_{V',r}^q$, this last sum becomes
$$(1-t)^{\ell}\sum_{\sigma\in\mathcal{D}_{V',r}^q}(-t)^{\#\{\text{even fixed points of }\sigma\}}(1-t)^{\#\{\text{other cycles of }\sigma\}}.$$

If $\sigma \in \mathcal{D}_{V,r}^q$, we let $|\sigma|$ be the underlying (unlabeled) $D$-permutation in $\mathcal{D}_{V}$. We thus have

$$\chi_{\mathcal{L}(\mathcal{K}_{\nu}^q)}(t) = (-1)^{m+n-1}(1-t)^{\ell}\sum_{\sigma \in \mathcal{D}_{V'}}\sum_{\substack{\tau \in \mathcal{D}_{V',n}^q\\|\tau| = \sigma}}(-t)^{\#\{\text{even fixed points of }\sigma\}}(1-t)^{\#\{\text{other cycles of }\sigma\}}.$$

Since $2r \notin V'$, given any $\sigma \in \mathcal{D}_{V'}$ we can construct a $\tau \in \mathcal{D}_{V',r}^q$ with $|\tau| = \sigma$ by labeling the largest entry of each cycle of $\sigma$ with $0$ and freely labeling all of the entries of $\sigma$ with the elements of $\{0,\dots,q-1\}$. This means that there are exactly $q^{|V'| - \#\{\text{cycles of }\sigma\}}$ many $\tau\in\mathcal{D}_{V',r}^q$ with $|\tau| = \sigma$.

Letting $\text{cyc}(\sigma)$ be the number of cycles of $\sigma$, this means that we have
\begin{align*}
\chi_{\mathcal{L}(\mathcal{K}_{\nu}^q)}(t) &= (-1)^{m+n-1}(1-t)^{\ell}\sum_{\sigma \in \mathcal{D}_{V'}}q^{|V'| - \text{cyc}(\sigma)}(-t)^{\#\{\text{even fixed pts of }\sigma\}}(1-t)^{\#\{\text{other cycles of }\sigma\}}\\
&= (-1)^{m+n-1}(1-t)^{\ell}q^{|V'|}\sum_{\sigma \in \mathcal{D}_{V'}}\left(\frac{-t}{q}\right)^{\#\{\text{even fixed points of }\sigma\}}\left(\frac{1-t}{q}\right)^{\#\{\text{other cycles of }\sigma\}}\\
&= (-1)^{m+n-1}(1-t)^{\ell}q^{m+n-\ell - 1}\Lambda_{V}\left(\frac{1-t}{m},\frac{1-t}{m},1,0,\frac{-t}{m},1\right),
\end{align*}
as desired, with the last equality following from Lemma \ref{GeneralizedbijectionProp}.
\end{proof}

\subsection{Generating Function Proofs}
Since $\mathcal{H}_{n,k}^m = \mathcal{H}_{\nu}^m$ for $\nu = (\underbrace{k,k,\dots,k}_{n})$, we are now able to prove Theorem \ref{kStaircaseDowlingThm}.

\begin{proof}[Proof of Theorem \ref{kStaircaseDowlingThm}]
By Theorem \ref{DowlingDuFoTheorem}, we have
$$\chi_{\mathcal{L}(\mathcal{H}_{n,k}^m)}(t) = (-1)^{nk+n-1}(1-t)^{k}m^{nk+n-k-1}\Lambda_{S_n^k}\left(\frac{1-t}{m},\frac{1-t}{m},1,0,\frac{-t}{m},1\right),$$
where 
$$S_n^k =\{1,3,\dots,2nk-1\}\sqcup\{2k,4k,\dots,2nk\}.$$

We have
\begin{align*}
\sum_{n\geq 1}\chi_{\mathcal{L}(\mathcal{H}_{n,k}^m)}(t)u^{n-1} &= \sum_{n\geq 1} (-1)^{nk+n-1}(1-t)^{k}m^{nk+n-k-1}\Lambda_{S_n^k}\left(\frac{1-t}{m},\frac{1-t}{m},1,0,\frac{-t}{m},1\right)u^{n-1}\\
&= (t-1)^k\sum_{n\geq 1}\Lambda_{S_n^k}\left(\frac{1-t}{m},\frac{1-t}{m},1,0,\frac{-t}{m},1\right)((-m^{k+1})u)^{n-1}.
\end{align*}
By Theorem \ref{FixedStepGenFun}, this last sum is equal to
$$(t-1)^k\sum_{n\geq 1}\frac{\left(\frac{1-t}{m}\right)^{(n-1)}\left(\left(\frac{1-t}{m}+1\right)^{(n-1)}\right)^k((-m)^{n-1})^{k+1}u^{n-1}}{\prod_{i=0}^{n-1}\left(1-\left(\frac{1-t}{m}+i\right)^{k-1}\left[\left(\frac{1-t}{m}+i\right)\frac{-1}{m}- \left(\frac{1-t}{m}+i\right)i\right](-m)^{k+1}u\right)}.$$

Now,
\begin{align*}
(-m)^{n-1}\left(\frac{1-t}{m}\right)^{(n-1)} &= (-m)^{n-1}\prod_{j=1}^{n-1}\left(\frac{1-t}{m}+j-1\right)\\
&= \prod_{j=1}^{n-1}t-1 -(j-1)m\\
&= (t-1)_{n-1,m}.
\end{align*}

Similarly,
$$(t-1)(-m)^{n-1}\left(\frac{1-t}{m}+1\right)^{(n-1)} = (t-1)\prod_{j=1}^{n-1}(t-1-jm) = \prod_{j=0}^{n-1}(t-1-jm) = (t-1)_{n,m}.$$

Thus, the generating function formula above simplifies to
$$\sum_{n\geq 1}\chi_{\mathcal{L}(\mathcal{H}_{n,k}^m)}(t)u^{n-1} = \sum_{n\geq 1}\frac{(t-1)_{n-1,m}\left((t-1)_{n,m}\right)^ku^{n-1}}{\prod_{i=0}^{n-1}\left(1-(im+1)(t-(im+1))^ku\right)},$$
as desired.
\end{proof}

Since $\mathcal{J}_{n,k}^m = \mathcal{H}_{\nu}^m$ for $\nu = (k,\underbrace{0,\dots,0}_{n-1})$, a similar argument allows us to prove Theorem \ref{CompBipDowlingGenThm}.

\begin{proof}[Proof of Theorem \ref{CompBipDowlingGenThm}]
By Theorem \ref{DowlingDuFoTheorem}, we have
$$\chi_{\mathcal{J}_{n,k}^m}(t) = \begin{cases} (t-1)^k, & n=1\\ (-m)^{n+k-1}\Lambda_{T_n^k}\left(\frac{1-t}{m},\frac{1-t}{m},1,0,\frac{-t}{m},1\right), & n\geq 2 \end{cases},$$
so
\begin{align*}
\sum_{n\geq 2}\chi_{\mathcal{J}_{n,k}^m}(t)u^{n-1} &= (-m)^{k}\sum_{n\geq 2}\Lambda_{T_n^k}\left(\frac{1-t}{m},\frac{1-t}{m},1,0,\frac{-t}{m},1\right)(-mu)^{n-1},
\end{align*}
and by Theorem \ref{CompBipGenThm} the last sum is equal to
$$(-m)^k\sum_{n\geq 2}\frac{\left(\frac{1-t}{m}\right)^{(n-2)}\Lambda_{T_2^k}(x+n-2,\bar{x}+n-2)\left|_{\left(\frac{1-t}{m},\frac{1-t}{m},1,0,\frac{-t}{m},1\right)}\right.(-m)^{n-1}u^{n-1}}{\prod_{i=0}^{n-2}\left[1-\left(\frac{-1}{m}-i\right)(-m)u\right]}.$$

Now, we know that
\begin{align*}\Lambda_{T_2^k}(x+i,\bar{x}+i) = &(y+\bar{x}+i)(x+\bar{z}+i)^k\\ &+ (x+i)^{k-1}\left[(x+i)(\bar{y}-y) -k(\bar{x}+i)(\bar{z}-z) - (x+i)(\bar{x}+i)\right],
\end{align*}

so

\begin{align*}(-m)&^{k+1}\Lambda_{T_2^k}(x+n-2,\bar{x}+n-2)\left|_{\left(\frac{1-t}{m},\frac{1-t}{m},1,0,\frac{-t}{m},1\right)}\right.\\ &= (t-1-m(n-2))\left[(t-1-m(n-1))^k+(m(n-2)+1)(t-1-m(n-2))^{k-1}\right].
\end{align*}

Since $(-m)^{n-2}\left(\frac{1-t}{m}\right)^{(n-2)} = (t-1)_{n-2,m}$ and $(t-1-m(n-2))(t-1)_{n-2,m} = (t-1)_{n-1,m}$, the generating function formula simplifies to

\begin{align*}\sum_{n\geq 2} &\chi_{\mathcal{J}_{n,k}^m}(t)u^{n-1}\\ &= \sum_{n\geq 2}\frac{(t-1)_{n-1,m}\left[(t-1-m(n-1))^k+(m(n-2)+1)(t-1-m(n-2))^{k-1}\right]u^{n-1}}{\prod_{i=0}^{n-2}\left[1-(mi+1)u\right]},\end{align*}
as desired.

\end{proof}

\section{Final Remarks and Further Questions}

	In \cite[Corollary 3.9]{Type_A_Paper}, it was shown that $-\chi_{\mathcal{L}(\mathcal{H}_{2n-1})}(0)$ is equal to the \emph{(unsigned) Genocchi number} $g_n$, and Hetyei's count of the number of regions of $\mathcal{H}_{2n-1}$ in \cite{Alternation_Acyclic} tells us that $-\chi_{\mathcal{L}(\mathcal{H}_{2n-1})}(-1)$ is equal to the median Genocchi number $h_n$. It therefore seems reasonable to define a family of generalizations $g_{n,k}$ and $h_{n,k}$ by
	\begin{align}
		g_{n,k} &= (-1)^{nk+n-1}\chi_{\mathcal{L}(\mathcal{H}_{n,k})}(0)\\
		h_{n,k} &= (-1)^{nk+n-1}\chi_{\mathcal{L}(\mathcal{H}_{n,k})}(-1).
	\end{align}
	
	Using Corollary \ref{kCharPolyCor}, we can derive generating function formulas for these sequences as functions of $k$.
	
	\begin{align}
		\sum_{n\geq 1}g_{n,k}u^n = \sum_{n\geq 1}(-1)^{nk+n-1}\chi_{\mathcal{L}(\mathcal{H}_{n,k})}(0) &= \sum_{n\geq 1}\frac{(n-1)!(n!)^ku^n}{\prod_{i=1}^n(1+i^2u)}\\
		\sum_{n\geq 1}h_{n,k}u^n = \sum_{n\geq 1}(-1)^{nk+n-1}\chi_{\mathcal{L}(\mathcal{H}_{n,k})}(-1) &= \sum_{n\geq 1}\frac{n!((n+1)!)^ku^n}{\prod_{i=1}^n(1+i(i+1)u)}
	\end{align}
	
	These generating functions reduce to generating functions due to Barsky and Dumont \cite{Barsky_Dumont} for the Genocchi and median Genocchi numbers, respectively, when $k=1$.

Specializing the proof of Theorem \ref{DowlingDuFoTheorem} to $m=1$ and $V = S_n^k$ (c.f. \cite[Theorem 4.14]{Type_A_Paper}), we see that
\begin{equation}\label{kFactorEq}
	\chi_{\mathcal{L}(\mathcal{H}_{n,k})}(t) = (1-t)^k\sum_{\sigma \in \mathcal{D}_{S_{n-1}^k}}(-t)^{\#\{\text{even fixed points of }\sigma\}}(1-t)^{\#\{\text{other cycles of } \sigma\}}
\end{equation}

By evaluating Equation (\ref{kFactorEq}) at $t=-1$, we obtain a decomposition of $h_{n,k}$ into powers of $2$ for all $n\geq 2$:
\begin{equation}\label{kMedGenocchiDecomp}
h_{n,k} = \sum_{j=1}^{k(n-1)}h_{n-1,k}^j2^{j+k},
\end{equation}
where $h_{n,k}^j$ is the number of D-permutations on $S_n^k$ with exactly $j$ cycles that are not even fixed points.

In \cite[Corollary 4.16]{Type_A_Paper}, the authors obtain a decomposition of the Genocchi numbers into powers of $2$ that is expected to be the same as a decomposition due to Sundaram \cite[Theorem 3.15]{sund}. However, the proof of this decomposition relies on a factorization of $\chi_{\mathcal{L}(\mathcal{H}_{2n-1})}(t)$ \cite[Theorem 4.15]{Type_A_Paper} whose proof does not immediately generalize to $\mathcal{H}_{n,k}$.

\begin{question}
What is the largest power of $(t-1)$ that divides $\chi_{\mathcal{L}(\mathcal{H}_{n,k})}(t)$?
\end{question}

\begin{question}
Is there a decomposition of $g_{n,k}$ analogous to Sundaram's decomposition of $g_n$?
\end{question}

There is a wealth of literature studying the Genocchi numbers and median Genocchi numbers. It would be interesting to see which results about those sequences can be generalized to this broader setting. For example, the Genocchi and median Genocchi numbers can be obtained from one another via a triangular array known as a Seidel triangle (see, e.g., \cite{Further_Triangles}).

\begin{question}
Is there an analogous relationship between $g_{n,k}$ and $h_{n,k}$?
\end{question}
 
\section*{Acknowledgments}
This paper is based on work conducted during the author's dissertation research. As such, he would like to thank his thesis committee: Bruno Benedetti, Mitsunori Ogihara, Richard Stanley, and especially his advisor Michelle Wachs, without whose guidance this paper would not have been possible. The author also wishes to thank Vic Reiner for introducing him to the notion of Ferrers graphs, and Bennet Goeckner and Joseph Doolittle for their help with editing. 

\printbibliography
\end{document}